\documentclass{article}[11pt]
\usepackage{graphicx} % Required for inserting images

\usepackage{amsmath}% http://ctan.org/pkg/amsmath
\usepackage{amsthm}
\usepackage{amssymb}
\usepackage[export]{adjustbox}% http://ctan.org/pkg/adjustbox
\usepackage[margin=0.75in]{geometry}
\usepackage{setspace}
\usepackage{enumitem}
\usepackage{float}
\usepackage{tabularx,ragged2e}
\newtheorem{thm}{Theorem}[section]
\newtheorem{defn}[thm]{Definition}
\newtheorem{prop}[thm]{Proposition}

\newtheorem{cor}[thm]{Corollary}

\newtheorem{rem}[thm]{Remark}
\newtheorem{lem}[thm]{Lemma}
\newtheorem*{lem*}{Lemma}
\newtheorem*{thm*}{Theorem}
\newtheorem*{cor*}{Corollary}
\newtheorem*{rem*}{Remark}

\newtheorem*{clm*}{Claim}
\newtheorem{ques}{Question}
\newtheorem*{ques*}{Question}
\usepackage[all]{xy}
\usepackage{tikz-cd}
\usepackage{mathtools}
\usepackage{comment}
\usepackage{authblk}
\usepackage{url}
\newcommand{\N}{\mathbb{N}}

\newcommand{\girth}{\mathrm{girth}}

\usepackage{hyperref}
\hypersetup{
	colorlinks=true,
	linkcolor=red,
	citecolor=blue,
}

\newtheorem*{mainthm}{Main Theorem}

\newcommand\blfootnote[1]{%
  \begingroup
  \renewcommand\thefootnote{}\footnote{#1}%
  \addtocounter{footnote}{-1}%
  \endgroup
}

\title{Graphical small cancellation and hyperfiniteness of boundary actions}
\author{Chris Karpinski, Damian Osajda, Koichi Oyakawa}

\AtEndDocument{%
	\par
	\medskip
	\begin{tabular}{@{}l@{}}%
		\textsc{Chris Karpinski}\\
		Department of Mathematics and Statistics, McGill University, Burnside Hall, \\
		805 Sherbrooke Street West, \\
		Montreal, QC, H3A 0B9, Canada. \\
		\textit{E-mail address}: \texttt{christopher.karpinski@mail.mcgill.ca}
\end{tabular}}
\AtEndDocument{%
	\par
	\medskip
	\begin{tabular}{@{}l@{}}%
		\textsc{Damian Osajda}\\
		Institut for Matematiske Fag, University of Copenhagen,\\ 
        2100 Copenhagen, Denmark.
        \\
        Instytut Matematyczny,
	Uniwersytet Wroc\l awski,\\
	pl.\ Grun\-wal\-dzki 2/4,
	50--384 Wroc\-{\l}aw, Poland. \\
		\textit{E-mail address}: \texttt{dosaj@math.uni.wroc.pl}
\end{tabular}}
\AtEndDocument{%
	\par
	\medskip
	\begin{tabular}{@{}l@{}}%
		\textsc{Koichi Oyakawa}\\
		Department of Mathematics and Statistics, McGill University, Burnside Hall, \\
		805 Sherbrooke Street West, \\
		Montreal, QC, H3A 0B9, Canada. \\
		\textit{E-mail address}: \texttt{koichi.oyakawa@mail.mcgill.ca}
\end{tabular}}

\date{}

\begin{document}

\maketitle
\blfootnote{MSC (2020) classes: 20F65, 03E15}
\begin{abstract}
    We study actions of (infinitely presented) graphical small cancellation groups on the Gromov boundaries of their coned-off Cayley graphs. 
    %\textcolor{black}{Chris: I looked through various papers on graphical small cancellation groups and it seems there is no consistent terminology for the coned-off Cayley graph (most papers don't give it any name). So I would suggest keeping "coned-off"}-- natural hyperbolic Cayley graphs associated to such groups. 
    We show that a class of graphical small cancellation groups, including (infinitely presented) classical small cancellation groups, admit {hyperfinite} boundary actions, more precisely, the orbit equivalence relation that they induce on the boundaries of the coned-off Cayley graphs is hyperfinite. 
    %\DO{I have doubts about singular/plural used here; but I am not an expert in English} \textcolor{black}{Chris: I think you have a point. I changed "boundary" to "boundaries"}
\end{abstract}

\section{Introduction}

%\textcolor{black}{In descriptive set theory, one studies equivalence relations with various properties on Polish spaces, i.e.\ separable and completely metrizable topological spaces. Among the equivalence relations that one may equip on a Polish space, a particularly simple class are the \emph{hyperfinite} equivalence relations, which are equivalence relations that can be decomposed as a countable increasing union of equivalence relations all of whose equivalence classes are finite. Equivalently, these equivalence relations can be thought of as those whose equivalence classes can be ``re-arranged'' in a Borel manner to have the structure of lines. }

%\textcolor{black}{On the other hand, in the field of geometric group theory, one studies groups as geometric objects, transporting notions from differential geometry, such as negative curvature, to group theory. Groups with notions of negative curvature have a ``boundary at infinity'', a Polish space on which the group acts by homeomorphisms. A somewhat unexpected connection between descriptive set theory and geometric group theory has been the study of the orbit equivalence relations of ``negatively curved'' groups acting on their boundaries. Remarkably, for many ``negatively curved'' groups, this orbit equivalence relation turns out to be hyperfinite.}

The classification problem, that is, classifying mathematical objects up to an equivalence relation, is a fundamental question arising in all areas of mathematics. A naive approach to this problem is to try to find a complete invariant in a definable way (e.g. two $n \times n$ complex matrices are similar if and only if they have the same Jordan normal form). However, it is not always possible to classify objects by a definable complete invariant. For example, a famous fact in measure theory is that the equivalence relation on the unit interval $[0,1]$ defined by $a \sim b \iff a-b \in \mathbb{Q}$ doesn't admit a Borel subset containing exactly one element from each equivalence class.

Descriptive set theory provides methods and notions to measure the complexity of classification problems, or equivalence relations, even when they are more complicated than having complete invariants. In this paper, we study the complexity of the orbit equivalence relation of dynamical systems that arise in geometry and group theory by using one of these notions called hyperfiniteness, which is the next simplest complexity after the completely classifiable one. More specifically, when a group acts isometrically on a hyperbolic space, the action naturally extends to the Gromov boundary of the hyperbolic space. This action on the boundary has been extensively studied in geometric group theory along with its connection to other areas such as ergodic theory and operator algebras, and hyperfiniteness of boundary actions is a rapidly growing research topic due to its relation to problems surrounding Weiss's conjecture.

The question of the hyperfiniteness of boundary actions of groups has its roots in the work of Dougherty, Jackson and Kechris \cite{DJK94}, who showed that the tail equivalence relation on the space $\Omega^{\N}$ of infinite sequences in some countable alphabet $\Omega$ is hyperfinite. Using the description of the Gromov boundary of a free group as the set of infinite reduced words in the free generators, this immediately implies that the orbit equivalence relation of the action of a countably generated free group on its Gromov boundary is hyperfinite. 

The work of Dougherty, Jackson and Kechris was later generalized 
by Huang, Sabok and Shinko \cite{HSS19}, who showed that cubulated hyperbolic groups admit hyperfinite actions on their Gromov boundaries. This was finally proved for all hyperbolic groups by Marquis and Sabok \cite{MS20}, and recently a new proof was found by Naryshkin and Vaccaro \cite{NaryshVacc}. In \cite{KarRelHyp}, the case of relatively hyperbolic groups was treated, where it was shown that the action of relatively hyperbolic groups on their Bowditch boundaries is hyperfinite. In another direction, Przytycki and Sabok showed that mapping class groups of finite type orientable surfaces induce hyperfinite equivalence relations on the boundaries of the arc and curve graphs \cite{przytycki_sabok_2021}. \textcolor{black}{Furthermore, in \cite{Hyp_Tree_Actions} sufficient conditions were identified for an action of a countable group on a countable tree such that the induced orbit equivalence relation on the Gromov boundary of the tree is hyperfinite. These conditions include acylindricity of the action of the group on the tree.}

Substantial progress in the study of boundary actions of groups was achieved by Oyakawa \cite{MR4715159}, who showed that for every countable acylindrically hyperbolic group $G$, there exists a hyperbolic Cayley graph $\Gamma$ of $G$ with an acylindrical action of $G$ such that the orbit equivalence relation of $G$ acting on the Gromov boundary $\partial \Gamma$ is hyperfinite. 
%\DO{shouldn't we mention "Hyperfiniteness for group actions on trees" somewhere here?} \textcolor{black}{Chris: Added above}
The question remains of the further generalization to \emph{all} {\color{black}acylindrical actions on hyperbolic Cayley graphs} of acylindrically hyperbolic groups. 

\begin{ques}
    \label{A}
    Given a countable acylindrically hyperbolic group $G$, are the orbit equivalence relations of the actions of $G$ on the Gromov boundaries {\color{black}induced by acylindrical actions on its hyperbolic} Cayley graphs hyperfinite?
\end{ques}

(Classical) Small cancellation (see e.g.\ \cite{MR1086661}) is a powerful tool for constructing infinite groups. Its more general version -- the graphical small cancellation -- has been recently used for providing examples of groups with very exotic properties \cite{OsajdaMonsters,Osajda,Small_Cancellation_Burnside}.
Graphical small cancellation groups fit into the framework of boundary actions of acylindrically hyperbolic groups, since they are examples of acylindrically hyperbolic groups \cite{Gruber_Sisto} \textcolor{black}{(note that the results in \cite{Gruber_Sisto} don't require finiteness of the generating set, as explained in \cite{GruberPhDThesis})}. %\DO{don't we want here the published version of [5] from TAMS?} \textcolor{black}{Chris: I think Gruber's PhD thesis [5] is not published. The paper by Gruber in TAMS is a different paper that deals with graphical C(6) and C(7) groups and doesn't mention acylindrical hyperbolicity, so I think that's not the one we want.}
We consider graphical small cancellation groups whose underlying graph satisfies a property called \emph{extreme fineness} (see Definition \ref{def: extremely fine}). This class of groups includes all the classical small cancellation ones. %\DO{We give an affirmative answer to Question~\ref{A} in this setting.}\KO{I think the action of graphical small cancellation groups on its coned-off Cayley graph don't need to be acylindridal even in the case of classical small cancellation, so the main theorem doesn't exactly fit into the setting of Question~\ref{A}. Gruber-Sisto proved acylindridal hyperbolic of graphical small cancellation groups by showing that their action on the coned off Cayley graph has a WPD element, which is much weaker than acylindricity of the action.}
%\textcolor{black}{Motivated by the above-mentioned results on hyperfiniteness of boundary actions of groups with negative curvature features, we prove the following:}

%The main result of this paper is the following: 

\begin{mainthm}
    %\label{thm:main theorem}
    
    Let $G$ be a group defined by a graphical $C'(1/10)$ presentation $\langle S \vert \Gamma \rangle$ with $S$ countable and with the graph $\Gamma$ having countably many connected components $(\Gamma_n)_{n \in \N}$, all of which are finite. Suppose that $\Gamma$ is extremely fine. Let $Y$ denote the associated coned-off Cayley graph. Then the action $G \curvearrowright \partial Y$ is hyperfinite. 
\end{mainthm}

%\DO{$\Theta$ here and $\Gamma$ at the begining of Sec 2.3 are inconsistent; later in Sec 2.3 its $\Theta$ again} \textcolor{black}{Chris: Thanks! This should be fixed now.}

{\color{black}Note that neither $\Gamma$ nor $S$ needs to be finite, that is, $G$ needs not to be finitely presented, or even finitely generated. The coned-off Cayley graph $Y$ (Definition~\ref{def:coneCay}) is a natural hyperbolic Cayley graph on which a graphical small cancellation group acts acylindrically. While we need extreme fineness for our proof of the Main Theorem, along the way we develop an approach to boundary actions for general graphical small cancellation, which we believe will be useful for further studies.

Our proof builds off of methods used by Naryshkin--Vaccaro for hyperbolic groups \cite{NaryshVacc} and Oyakawa for acylindrically hyperbolic groups \cite{MR4715159}. The idea of the proof is to show that boundary points in $\partial Y$ can be represented by ``nice'' geodesic rays in the Cayley graph $X:= \mathrm{Cay}(G,S)$. For each boundary point $\xi \in \partial Y$, we obtain a ``bundle'' $G(\xi)$ of geodesic rays based at the identity vertex 1 in $X$ representing $\xi$. Using the small cancellation condition, we show that there exists a geodesic ray $\sigma_{\xi} \in G(\xi)$ with lexicographically least label in $S^{\N}$ (with respect to an arbitrary linear order on $S$). This allows us to construct a Borel injection $\Phi: \partial Y \to S^{\N}$ via $\Phi(\xi) = \sigma_{\xi}$ (where $S^{\N}$ is equipped with the product topology, using the discrete topology on $S$). Letting $E_t$ denote the tail equivalence relation on $S^{\N}$, $E_G$ denote the orbit equivalence relation of $G \curvearrowright \partial Y$, and letting $R_t' = \Phi^{-1}(E_t)$ be the pullback of $E_t$ to $\partial Y$ via $\Phi$ and putting $R_t = R_t' \cap E_G$, the extreme fineness property allows us to conclude that each $E_G$-class intersects only finitely many $R_t$-classes. The hyperfiniteness of $E_t$ (\cite[Corollary 8.2]{DJK94}) and the fact that each $E_G$-class intersects only finitely many $R_t$-classes implies by Proposition \ref{prop:JKL} that $E_G$ is hyperfinite. 

The hyperfiniteness of boundary actions of general graphical small cancellation groups (i.e.\ those whose defining graphs do not necessarily satisfy the extreme fineness property) remains open. 

\begin{ques}
    \label{B}
    Let $G$ be a group defined by a graphical $C'(1/10)$ presentation $\langle S \vert \Gamma \rangle$ with $S$ countable and with the graph $\Gamma$ having countably many connected components $(\Gamma_n)_{n \in \N}$, all of which are finite. Let $Y$ denote the associated coned-off Cayley graph. Is the action $G \curvearrowright \partial Y$ hyperfinite?
\end{ques}

%A negative answer to Question \ref{B} would yield a negative answer to Question \ref{A}, since, while the action $G \curvearrowright Y$ need not be acylindrical in general (\cite[Example 3.10]{Gruber_Sisto}), it is always possible modify the graphical presentation to obtain a $C'(1/10)$ presentation such that the action $G \curvearrowright Y$ is non-elementary acylindrical. 

\textbf{Acknowledgements}: We are grateful to Antoine Poulin, Marcin Sabok and Andrea Vaccaro for reading previous drafts of the paper and offering several helpful comments. \textcolor{black}{We thank the anonymous referees for the many helpful comments and useful feedback.} D.O.\ was partially supported by the Carlsberg Foundation, grant CF23-1225.

\section{Preliminaries}

\subsection{Gromov hyperbolic metric spaces and their boundaries}

In this section, we review the Gromov boundary of a hyperbolic space; readers are referred to \cite{BH99} for more.

\begin{defn}\label{def:Gromov product}
    Let $(S,d_S)$ be a metric space. For $x,y,z\in S$, we define the \textbf{Gromov product} $(x,y)_z$ by
\begin{align}\label{eq:gromov product}
    (x,y)_z=\frac{1}{2}\left( d_S(x,z)+d_S(y,z)-d_S(x,y) \right).    
\end{align}
\end{defn}

\begin{prop} \label{prop:hyp sp}
    For any geodesic metric space $(S,d_S)$, the following conditions are equivalent.
    \item[(1)]
    There exists $\delta \geq 0$ satisfying the following property. Let $x,y,z\in S$, and let $p$ be a geodesic path from $z$ to $x$ and $q$ be a geodesic path from $z$ to $y$. If two points $a\in p$ and $b\in q$ satisfy $d_S(z,a)=d_S(z,b)\le (x,y)_z$, then we have $d_S(a,b) \le \delta$.
    \item[(2)] 
    There exists $\delta \geq 0$ such that for any $w,x,y,z \in S$, we have
    \[
    (x,z)_w \ge \min\{(x,y)_w,(y,z)_w\} - \delta.
    \]
\end{prop}

\begin{defn}
    A geodesic metric space $S$ is called \textbf{hyperbolic}, if $S$ satisfies the equivalent conditions (1) and (2) in Proposition \ref{prop:hyp sp}. We call a hyperbolic space $\delta$-\textbf{hyperbolic} with $\delta \geq 0$, if $\delta$ satisfies both of (1) and (2) in Proposition \ref{prop:hyp sp}. A connected graph $\Gamma$ is called \textbf{hyperbolic}, if the geodesic metric space $(\Gamma,d_\Gamma)$ is hyperbolic, where $d_\Gamma$ is the graph metric of $\Gamma$.
\end{defn}

In the remainder of this section, suppose that $(S,d_S)$ is a hyperbolic geodesic metric space.

\begin{defn}\label{def:seq to infty}
    A sequence $(x_n)_{n=1}^\infty$ of elements of $S$ is said to \textbf{converge to infinity}, if we have $\lim_{i,j\to\infty} (x_i,x_j)_o =\infty$ for some (equivalently any) $o\in S$. For two sequences $(x_n)_{n=1}^\infty,(y_n)_{n=1}^\infty$ in $S$ converging to infinity, we define the relation $\sim$ by $(x_n)_{n=1}^\infty \sim (y_n)_{n=1}^\infty$ if we have $\lim_{i,j\to\infty} (x_i,y_j)_o =\infty$ for some (equivalently any) $o\in S$.
\end{defn}

\begin{rem}
    It's not difficult to see that the relation $\sim$ in Definition \ref{def:seq to infty} is an equivalence relation by using the condition (2) of Proposition \ref{prop:hyp sp}.
\end{rem}

\begin{defn}
    The quotient set $\partial S$ is defined by
    \[
    \partial S = \{ {\rm sequences~in~ }S {\rm ~converging~to~infinity} \} / \sim
    \]
    and called \textbf{Gromov boundary} of $S$.
\end{defn}

\begin{rem}
    The set $\partial S$ is sometimes called the sequential boundary of $S$. Note that $\partial S$ sometimes coincides with the geodesic boundary of $S$ (e.g.\ when $S$ is a proper metric space), but this is not the case in general.
\end{rem}

By \cite[Proposition III.H.3.21]{BH99}, $\partial S$ has a natural metrizable topology, which is compact if $S$ is a proper metric space (i.e.\ when closed balls in $S$ are compact). 

\subsection{Descriptive set theory}

In this section, we review concepts in descriptive set theory.

\begin{defn}\label{def:Polish}
    A \textbf{Polish space} is a separable completely metrizable topological space.
\end{defn}

By \cite[Lemma 4.1]{MR4715159}, if $\Gamma$ is a hyperbolic graph with countable vertex set (with edges assigned length 1), then the Gromov boundary $\partial \Gamma$ with its natural topology is a Polish space. In particular, if $\Gamma$ is a hyperbolic Cayley graph of a countable group (with respect to a possibly infinite generating set), then $\partial \Gamma$ is a Polish space. 

\begin{defn}
    A measurable space $(X,\mathcal{B})$ is called a \textbf{standard Borel space}, if there exists a topology $\mathcal{O}$ on $X$ such that $(X,\mathcal{O})$ is a Polish space and $\mathcal{B}(\mathcal{O})=\mathcal{B}$ holds, where $\mathcal{B}(\mathcal{O})$ is the $\sigma$-algebra on $X$ generated by $\mathcal{O}$.
\end{defn}

\begin{defn}
    Let $X$ be a standard Borel space and $E$ be an equivalence relation on $X$. $E$ is called \textbf{Borel} if $E$ is a Borel subset of $X\times X$. $E$ is called \textbf{countable} (resp.\ \textbf{finite}), if for any $x\in X$, the set $\{ y\in X \mid (x,y)\in E \}$ is countable (resp.\ finite).
\end{defn}

\begin{rem}
    The word ``countable Borel equivalence relation" is often abbreviated to ``CBER".
\end{rem}

If $G$ is a countable group acting by Borel automorphisms on a standard Borel space $X$, then the \textbf{orbit equivalence relation} $E_G$ defined by
     \[
        (x,y) \in E_G \iff \exists\ g\in G {\rm ~s.t.~} gx=y
     \]
is a CBER. In fact, by the classical Feldman--Moore theorem \cite[Theorem 1.3]{KM04}, every CBER arises in this way. 

In this paper, we will be interested in studying a property of CBERs known as \emph{hyperfiniteness}. 

\begin{defn}
    Let $X$ be a standard Borel space. A countable Borel equivalence relation $E$ on $X$ is called \textbf{hyperfinite}, if there exist finite Borel equivalence relations $(E_n)_{n=1}^\infty$ on $X$ such that $E_n\subset E_{n+1}$ for any $n\in\N$ and $E=\bigcup_{n=1}^\infty E_n$.
\end{defn}

Recall that any countable set $\Omega$ with the discrete topology is a Polish space. Hence, $\Omega^\N$ endowed with the product topology is a Polish space. 

\begin{defn}\label{def:tail equivalence}
        Let $\Omega$ be a countable set. The equivalence relation $E_t(\Omega)$ on $\Omega^\N$ is defined as follows: for $w_0=(s_1,s_2,\ldots), w_1=(t_1,t_2,\ldots) \in \Omega^\N$,
    \[
    (w_0,w_1)\in E_t(\Omega) \iff \exists n, \exists m \in\N\cup\{0\} {\rm ~s.t.~}\forall i\in\N, s_{n+i}=t_{m+i}.
    \]
    $E_t(\Omega)$ is called the \textbf{tail equivalence relation} on $\Omega^\N$.
\end{defn}

When the set $\Omega$ is understood, we will often write $E_t$ for $E_t(\Omega)$. 

Proposition \ref{prop:DJK} below is central to the proof of our main theorem and to the proof of all previous results concerning hyperfiniteness of boundary actions of groups. 
%rests on the following We list some facts needed for the proof of the main theorem. 
It is a particular case of \cite[Corollary 8.2]{DJK94}.

\begin{prop}\label{prop:DJK}{\rm (cf. \cite[Corollary 8.2]{DJK94})}
For any countable set $\Omega$, the tail equivalence relation $E_t(\Omega)$ on $\Omega^\N$ is a hyperfinite CBER.
\end{prop}

The following result will also play a key role in the proof of the main theorem.

\begin{prop}\label{prop:JKL}{\rm \cite[Proposition 1.3.(vii)]{JKL02}}
    Let $X$ be a standard Borel space and $E,F$ be countable Borel equivalence relations on $X$. If $E\subset F$, $E$ is hyperfinite, and every $F$-equivalence class contains only finitely many $E$-classes, then $F$ is hyperfinite.
\end{prop}

\subsection{Graphical small cancellation theory}
\label{section: graphical small cancellation}
We follow closely the presentation given in \cite{Small_Cancellation_Burnside}. Throughout the paper, we allow graphs with loops and multi-edges.

\begin{defn}
    Let $S$ be a set. An \textbf{$S$-labeled graph} is a graph $\Gamma = (V,E)$ together with a map $E \to S \coprod S^{-1}$ which is compatible with the orientation of edges (i.e.\ the inversely oriented edge is assigned the inverse label). 
\end{defn}

Given an $S$-labeled graph $\Gamma$, we will denote the label of an edge path $\gamma$ (that is, the word in $S \coprod S^{-1}$ read from labels along $\gamma$) by $\ell(\gamma)$. 
%\DO{somehow $\ell(\gamma)$ immediately makes me think of the length of $\gamma$. Maybe $L(\gamma)$ or $\Lambda(\gamma)$? \textcolor{black}{Chris: $\ell$ is the notation for labels of paths used in Gruber-Sisto, so I think we should stick with it.}
We will also denote $\vert \gamma \vert$ the length of $\gamma$, i.e.\ the number of edges on $\gamma$. From the labeled graph $\Gamma$, we can define a presentation for a group $G(\Gamma)$, called the \textbf{graphical presentation} associated to $\Gamma$: 

$$G(\Gamma) = \langle S \vert \ell(\gamma) : \text{$\gamma$ is a simple closed path in $\Gamma$} \rangle$$
where a path is defined to be \emph{simple} if it does not self-intersect except possibly at its endpoints, and \emph{closed} if its endpoints are the same. 

\begin{defn}
    A \textbf{piece} in a labeled graph $\Gamma$ is a path $p \subset \Gamma$ such that there exists a path $q \subset \Gamma$ with $\ell(p) = \ell(q)$ and there is no label-preserving graph isomorphism $\phi : \Gamma \to \Gamma$ with $\phi(p) = q$. 
\end{defn}

\begin{defn}
    For $\lambda > 0$, a labeled graph $\Gamma$ satisfies \textbf{the graphical $C'(\lambda)$ small cancellation condition} if no two edges with the same initial vertex have the same label and for each piece $p$ contained in a connected component $\Gamma_i$ of $\Gamma$, we have $\vert p \vert < \lambda girth(\Gamma_i)$, where $girth(\Gamma_i) = \min \{\vert \gamma \vert: \gamma \subset \Gamma_i \text{ is a simple closed path}\}$ is the \emph{girth} of $\Gamma_i$. The girth of a tree is defined to be infinity. 
\end{defn}

Note that we can assume that any two edges in $\Gamma$ with common initial vertex have different labels by gluing together any two edges with a common initial vertex and same label. We will refer to the group defined by a labeled graph satisfying the graphical $C'(\lambda)$ small cancellation condition as a $C'(\lambda)$ \textbf{graphical small cancellation group}. Note that \emph{classical} small cancellation presentations and groups are precisely those arising from each connected component of $\Gamma$ being a cycle. Note also that our $C'(\lambda)$ small cancellation condition is stronger than the $Gr'(\lambda)$ small cancellation condition in \cite[Definition 2.3]{Gruber_Sisto} and is different from the $C'(\lambda)$ condition in \cite[Page 5]{Gruber_Sisto}, but is the same condition as in \cite{OsajdaMonsters}.

Recall that the \textbf{Cayley graph} of a group with respect to a generating set $S$ is the graph $\mathrm{Cay}(G,S)$ having vertex set $G$ and edge set $G \times S$, with an edge $(g,s)$ joining the vertices $g$ and $gs$. Denote $G = G(\Gamma)$ and $X = \mathrm{Cay}(G, S)$. Given any vertex $v$ in a connected component $\Gamma_i$ of $\Gamma$ and $g \in G$, there is a unique label preserving graph homomorphism $f_{v, g}: \Gamma_i \to X$ sending $v$ to $g$. The following lemma is implied by \cite[Lemma 2.15]{Gruber_Sisto}.

\begin{lem}
    If $G$ is a $C'(1/6)$ graphical small cancellation group, then for every vertex $v \in V(\Gamma_i)$ and every $g \in G$, the map $f_{v,g}$ above is an isometric embedding of $\Gamma_i$ into $X$, whose image is a convex subgraph of $X$. 
\end{lem}

We call the embedded connected components $\Gamma_i$ in $X$ the \textbf{relators} in $X$.  A \textbf{contour} is a simple closed path in $X$ contained in a relator.

From now on, we let $G = G(\Gamma)$ be a $C'(1/6)$ graphical small cancellation group, corresponding to a labeled graph $\Gamma$ over a countable labeling set, with finite connected components $(\Gamma_n)_{n \in \N}$.

We will now define a hyperbolic Cayley graph of $G$ obtained by coning off relators in $X$. Below, for an element $g \in G$ and a word $w$ over the generators $S$ of $G$, we will denote $g =_G w$ to denote that $g$ is represented by $w$ in $G$. 

\begin{defn}
    \label{def:coneCay}
    Let $W = \{g \in G : g =_G \ell(p) \text{ for some path } p \subset \Gamma\}$ be the set of all group elements in $G$ represented by the label of a path in $\Gamma$. The Cayley graph $Y := \mathrm{Cay}(G, S \cup W)$ is called the \textbf{coned-off} Cayley graph of $G$. 
\end{defn}

Let us remark here that the term ``coned-off'' is not standard and we just use it for the current paper. Usually, coning-off is obtained by adding a new vertex -- the apex. We do not add any new vertices (but relators indeed give rise to cones over each of its vertices). 

By \cite[Theorem 3.1]{Gruber_Sisto}, for $C'(1/6)$ graphical small cancellation groups (over possibly infinite generating sets), the coned off Cayley graph $Y$ is always hyperbolic. We can think of $Y$ as being obtained from $X$ by replacing every relator $\Theta$ in $X$ by the complete subgraph on its vertices. This replacement yields a metric space that is quasi-isometric to the Cayley graph $\mathrm{Cay}(G, S \cup W)$.

Intuitively, the metric $d_Y$ on $Y$ counts how many relators any geodesic in $X$ between two given points passes through. Geodesic paths in $Y$ correspond to ``geodesic sequences'' of relators. 

\begin{defn}
    A sequence $\Theta_1,\ldots,\Theta_n$, where each $\Theta_i$ is either a relator or an edge of $X$ not occurring on relators, is \textbf{geodesic} if $\Theta_i \cap \Theta_{i+1} \neq \emptyset$ for each $i = 0,\ldots,n-1$ and there does not exist another such sequence $\Theta_1' = \Theta_1,\ldots, \Theta_k' = \Theta_n$ with $k < n$. 
\end{defn}

\begin{prop}(\cite[Proposition 3.6]{Gruber_Sisto})
    \label{Geod decomp}
    Let $x \neq y$ be vertices in $X$, and let $\gamma$ be a geodesic in $X$ from $x$ to $y$. Denote $k = d_Y(x,y)$. Then $k$ is the minimal number such that $\gamma = \gamma_{1} \cdots \gamma_{k}$, where each $\gamma_{i}$ is either a path in some relator in $X$ or an edge in $X$ not occurring on any relator. %single edge not labeled by an element of $S$ that does not occur on $\Gamma$ or such that $\ell(\gamma_{i})$ labels a subpath of $\Gamma$. 
\end{prop}

In the notation of Proposition \ref{Geod decomp}, denoting $p_-$ and $p_+$ the initial and terminal vertices of a path $p$, this means $((\gamma_{1})_-, (\gamma_{1})_+)((\gamma_{2})_-, (\gamma_{2})_+) \cdots ((\gamma_{k})_-, (\gamma_{k})_+)$ is a geodesic edge path in $Y$ from $x$ to $y$. 

Proposition \ref{Geod decomp} says that each geodesic path $\gamma_X \subset X$ between vertices $x$ and $y$ of $X$ can be covered by a geodesic sequence $(\Theta_i)_{i=1}^k$ of length $k = d_Y(x,y)$, where each $\Theta_i$ is either a relator or an edge in $X$ not occurring on $\Gamma$. Hence any such geodesic $\gamma_X$ in $X$ projects to a geodesic $\hat{\gamma_X} = e_1 \cdots e_k$ in $Y$ from $x$ to $y$ with $e_i$ having lifts in $X$ contained in $ \Theta_i$.

We record an elementary but useful property of relators that we will need in the sequel.

\begin{lem}
    \label{geodesic subpath}
    Let $p$ be a geodesic in $X$ and let $\Theta \subset X$ be a relator. Then $\Theta \cap p$ is a subpath of $p$.
\end{lem}

\begin{proof}
    Let $p_-$ be the first vertex of $p$ in $\Theta$ and let $p_+$ be the last vertex of $p$ in $\Theta$. By \cite[Lemma 2.15]{Gruber_Sisto}, we have that $\Theta$ is a convex subgraph of $X$, and hence the segment of $p$ between $p_-$ and $p_+$ is in $\Theta$. Thus, $\Theta \cap p$ is the subsegment of $p$ between $p_-$ and $p_+$. 
\end{proof}

\textcolor{black}{We will now discuss terminology associated with disk diagrams over graphical presentations, following \cite{Gruber_Sisto}. 
\textcolor{black}{A \textbf{van Kampen diagram} $\Delta $ over a graphical presentation coming from an $S$-labeled graph $\Gamma$ is a compact, contractible, planar 2-complex homeomorphic to a disk whose 1-cells are labeled by elements of $S$ and whose 2-cells have boundary labels that are the labels of simple closed paths in $\Gamma$. The 1-cells of a disk diagram $\Delta$ are called \textbf{edges} and the 2-cells of $\Delta$ are called \textbf{faces}. We will assume that all of our van Kampen diagrams are \emph{reduced} in the sense of \cite{Gruber_Sisto}}. Note that minimal diagrams as in \cite[Definition 1.21]{GruberPhDThesis} are reduced.} %A path in a disk diagram $\Delta$ is \textbf{interior} if it has an edge which is not on the topological boundary of $\Delta$ in $\R^2$.} 

%\textcolor{black}{Now consider a van Kampen diagram $\Delta$ over a graphical presentation from an $S$-labeled graph $\Gamma$. Let $p$ be an edge path of $\Delta$ which is at the intersection of two 2-cells $\Pi$ and $\Pi'$ of $\Delta$ (so $p \subset \partial \Pi \cap \partial \Pi'$). By definition of a van Kampen diagram, each 2-cell $\Pi$ and $\Pi'$ induces a lift of $p$ to $\Gamma$ from lifts of $\partial \Pi$ and $\partial \Pi'$ to $\Gamma$. We say that $p$ \textbf{originates} from $\Gamma$ if there exist 2-cells $\Pi$ and $\Pi'$ of $\Delta$ which induce the same lift of $p$ to $\Gamma$. A van Kampen diagram $\Delta$ is called \textbf{reduced} if no edge path in $\Delta$ originates from $\Gamma$. We can always assume our van Kampen diagrams are reduced by merging any 2-cells $\Pi$ and $\Pi'$ such that $\partial \Pi \cap \partial \Pi'$ originates from $\Gamma$ into a single 2-cell obtained by gluing $\Pi$ and $\Pi'$ together along $\partial \Pi \cap \partial \Pi'$}. 

The following classification of simple geodesic bigons and triangles in Cayley graphs of small cancellation groups due to Strebel will be used several times throughout this paper. Note that the classification below is stated for \emph{classical} small cancellation groups, however, it applies equally well to our setting of graphical small cancellation groups, see \cite[Remark 3.11]{GruberPhDThesis}.

\begin{thm}(\cite[Theorem 43]{MR1086661})
    \label{thm: Strebel's bigon and triangle classification}
    Let $G = G(\Gamma)$ be a group defined by a $C'(\lambda)$ labeled graph $\Gamma$ for $\lambda \leq 1/6$ and $X$ its Cayley graph. Let $\Delta$ be a reduced van Kampen diagram over the graphical presentation $G(\Gamma)$. 

    \begin{enumerate}
        \item If $\Delta$ is the van Kampen diagram of a simple geodesic bigon with distinct vertices in $X$ and if $\Delta$ has more than one 2-cell, then $\Delta$ has shape $I_1$ as in Figure \ref{fig:shape I_1 diagram}.
        \item If $\Delta$ is the van Kampen diagram of a simple geodesic triangle with three distinct vertices in $X$ and if $\Delta$ has more than one 2-cell, then $\Delta$ has one of the forms shown in Figure \ref{fig:triangle classification shapes}.
    \end{enumerate}
\end{thm}

\begin{figure}[H]
    \centering
    \includegraphics[width=0.5\linewidth]{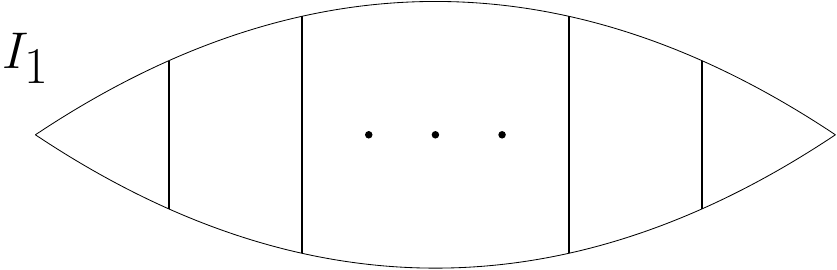}
    \caption{The shape of a van Kampen diagram bounding a simple geodesic bigon (shape $I_1$ in the terminology of \cite{MR1086661}).}
    \label{fig:shape I_1 diagram}
\end{figure}

\begin{figure}[h]
    \centering
    \includegraphics[width=0.7\linewidth]{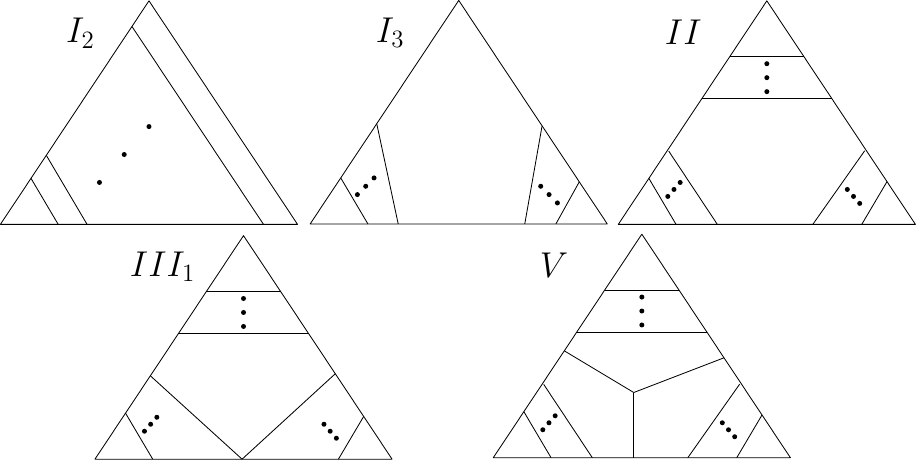}
    \caption{Possible shapes of van Kampen diagrams bounded by simple geodesic triangles (using the terminology of \cite{MR1086661}).}
    \label{fig:triangle classification shapes}
\end{figure}

\section{Hyperfiniteness of the action on the coned-off Cayley graph boundary}

\subsection{Geodesic representatives of coned-off Cayley graph boundary points}

Fix a graphical $C'(\lambda)$ small cancellation presentation $G = \langle S \cup S^{-1} \vert \Gamma_n : n \in \N \rangle$ with $S$ countable and each $\Gamma_n$ a finite connected graph, where $\lambda$ satisfies $\frac12 - 2 \lambda \geq 3 \lambda$ (i.e.\ $\lambda \leq \frac1{10})$. Denote $\Gamma = \coprod_{n \in \N} \Gamma_n$. We will assume that there is no label-preserving graph isomorphism of $\Gamma$ mapping one component $\Gamma_n$ to another $\Gamma_m$. If this were the case, then we could remove $\Gamma_m$ from $\Gamma$ and this would not change the graphical presentation.  

Let $X := \mathrm{Cay}(G,S)$ and let $Y = \mathrm{Cay}(G, S \cup W)$, where $W$ is the set of all words that can be read on paths in $\Gamma$. Recall that $Y$ is hyperbolic by \cite[Theorem 3.1]{Gruber_Sisto}, and thus by \cite[Lemma 4.1]{MR4715159}, $\partial Y$ is a Polish space. Denote $d_X$ (respectively, $d_Y$), the graph metric on $X$ (respectively, $Y$). For an edge path $p$ without self-intersections and vertices $a,b$ on $p$, we will denote $p_{[a,b]}$ the subsegment of $p$ between $a$ and $b$. Also, for an infinite edge path $p$ without self-intersections and a vertex $a \in p$, we will denote $p_{[a, \infty)}$ the terminal subray of $p$ starting at the vertex $a$.

In this section, we prove the Main Theorem.
%that under the assumption that the graph $\Gamma$ satisfies a condition called \emph{extreme fineness}, the action $G \curvearrowright \partial Y$ is hyperfinite. 
%\begin{thm}
%    \label{hyperfiniteness theorem}
%    If $\Gamma$ is extremely fine (see Definition \ref{def: extremely fine}), then the action of $G$ on $\partial Y$ induced by the natural action $G \curvearrowright Y$ is hyperfinite.
%\end{thm}
Its proof rests on the following key lemmas (Lemma \ref{lem:geod ray converging in Y} and Lemma \ref{lemma: geodesic representatives for boundary}), which assert that points in $\partial Y$ can be represented by geodesic rays in $X$. Note that in the following lemmas, we do not need to assume that $\Gamma$ is extremely fine (Definition~\ref{def: extremely fine}).  We will only need extreme fineness in the proof of Proposition \ref{prop: finite index}. 

{\color{black} We start by introducing nice geodesic rays in $X$ that represent points in $\partial Y$.
\begin{lem}\label{lem:geod ray converging in Y}
    Let $p=(p(0),p(1),p(2),\ldots)$ be a geodesic ray in $X$, where $p(i)$'s are vertices in $X$ composing $p$. The following conditions (1)-(3) are equivalent.
    \begin{itemize}
        \item[(1)]
        There exists $\xi \in \partial Y$ such that the sequence $(p(n))_{n \in \N}$ converges to $\xi$.
        \item[(2)]
        $\sup_{n \in \N}d_Y(p(0),p(n))= \infty$.
        \item[(3)] 
        There exist $\xi \in \partial Y$ and a sequence of vertices $(a_n)_{n \in \N}$ on $p$ such that (a) and (b) below are satisfied.
    
    \begin{itemize}
        \item[(a)] The subpath of $p$ defined by $r_i = p_{[a_{i-1}, a_i]}$ is either an edge in $X$ not appearing in any relator $\Theta$ or a subpath of $p$ contained in a relator (i.e.\ $r_i$ projects to an edge in $Y$).
        \item[(b)] $a_n \to \xi$ in $\partial Y$ and $(a_n)_n$ defines a geodesic ray in $Y$, i.e.\ $d_Y(a_i, a_j) = \vert i - j \vert$ for all $i,j \in \N$. %the sequence $(\hat{r_i})_{i \in \N}$ of edges in $Y$ defines a geodesic ray in $Y$. 
    \end{itemize}
    \end{itemize}
\end{lem}

\begin{proof}
    $\underline{(3)\Rightarrow(1)}$ We have $\lim_{n \to \infty}p(n) = \xi$ by $\lim_{i \to \infty}a_i = \xi$ and $\forall\,i \in \N,\,\forall\, v \in p_{[a_{i-1}, a_i]},\,d_Y(v,a_i) \le 1$.
    
    $\underline{(1)\Rightarrow(2)}$ This follows by $\lim_{n \to \N}d_Y(p(0),p(n))= \infty$.
    
    $\underline{(2)\Rightarrow(3)}$ Note that $\{d_Y(p(0),p(n))\}_{n \in \N}$ is non-decreasing by Proposition \ref{Geod decomp}, hence, $\sup_{n \in \N}d_Y(p(0),p(n))= \infty \iff \lim_{n \to \N}d_Y(p(0),p(n))= \infty$. By Proposition \ref{Geod decomp}, for each $n \in \N$, there exists a sequence of vertices $(a_{n,i})_{i = 0}^{k_n}$ on $p_{[p(0),p(n)]}$, where we have $k_n = d_Y(p(0),p(n))$, $a_{n,0}=p(0)$, and $a_{n,k_n}=p(n)$, such that the subpath $p_{[a_{n,i-1}, a_{n,i}]}$ is either an edge in $X$ not appearing in any relator $\Theta$ or a subpath of $p$ contained in a relator. The sequence $(a_{n,i})_{i = 0}^{k_n}$ is a geodesic path in $Y$. Hence, for any $i \in \N$ and any $m,n\in \N$ with $\min\{d_Y(p(0),p(n)), d_Y(p(0),p(m))\} \ge i+1$, we have $a_{m,i} \in p_{[p(0),a_{n,i+1}]}$. Indeed, $a_{m,i} \notin p_{[p(0),a_{n,i+1}]}$ implies $d_Y(p(0),a_{n,i+1}) \le i$, which contradicts $d_Y(p(0),a_{n,i+1}) = i+1$. 
    
    Thus, $\#\{a_{n,i} \mid n \in \N\} < \infty$ for any $i \in \N$. By taking subsequences and diagonal argument, there exist a sequence of vertices $(a_i)_{i \in \N}$ on $p$ and a subsequence $(n_j)_{j \in \N} \subset \N$ such that for any $i,j \in \N$ with $i \le j$, we have $a_{n_j,\,i} = a_i$. Hence, the sequence $(a_i)_{i \in \N}$ satisfies condition (a) and is a geodesic ray in $Y$. This implies condition (b) as well.
\end{proof}

\begin{defn}
    Given $\xi \in \partial Y$, we will say that a geodesic ray $p=(p(n))_{n = 0}^\infty$ in $X$ satisfying $\lim_{n \to \infty} p(n) = \xi$ \emph{represents} $\xi$ in $Y$. We will denote $G(\xi)$ the set of all geodesic rays in $X$ from $1 \in G$ representing $\xi$. We have a natural injective map $G(\xi) \to S^{\N}$ defined by sending a geodesic ray $p \in G(\xi)$ to its label $\mathrm{lab}(p) = (p(n-1)^{-1}p(n))_{n \in \N} \in S^{\N}$. Equipping $S$ with the discrete topology and $S^{\N}$ with the product topology, this induces a topology on $G(\xi)$ as a subspace of $S^{\N}$.
\end{defn}
}

\begin{lem}
    \label{lemma: geodesic representatives for boundary}
    For any $\xi \in \partial Y$, we have $G(\xi) \neq \emptyset$.
    
\end{lem}

\begin{proof}
    Let $(x_n)_n$ be a sequence of elements of $G$ representing $\xi$. For each $n$, let $p_n$ be a geodesic segment in $X$ from 1 to $x_n$. {\color{black} In the following, the Gromov product (see Definition \ref{def:Gromov product}) is always in $Y$.}

    Fix $n \in \N$ and a hyperbolicity constant $\delta$ for $Y$. Choose $i \in \N$ sufficiently large such that for all $j > i$, we have $(x_i, x_j)_1 > n + \delta + 2$ (such $i$ exists since $(x_n)_n$ converges to infinity). 

    We will show that the set $\mathcal{A}_n = \{v \in p_j : d_Y(1,v) = n \text{ and } j > i\}$ is finite. Note indeed that for each $j > i$, there exists a vertex $v \in p_j$ such that $d_Y(1,v) = n$, since by Proposition \ref{Geod decomp}, $p_j$ can be decomposed as $p_j = p_{j,1} p_{j,2} \cdots p_{j,k}$, where $k = d_Y(1,x_j) \geq (x_i, x_j)_1 > n$ and each $p_{j,m}$ projects to an edge in $Y$, so there exists a vertex $v \in p_{j,n-1} \cup p_{j,n} \cup p_{j, n+1}$ with $d_Y(1,v) = n$. 

    Let $p_i = p_{i,1} p_{i,2} \cdots p_{i,l}$ be a decomposition of $p_i$ as in Proposition \ref{Geod decomp} with each $p_{i,k}$ either an edge in $X$ that does not occur in $\Gamma$ or a subpath of $p_i$ contained in a relator. For each \textcolor{black}{$k = 1,\ldots,l$}, denote $\Theta_k$ a relator containing $p_{i,k}$ if such a relator exists or $\Theta_k = p_{i,k}$ otherwise. 

    To prove that $\mathcal{A}_n$ is finite, we will show that $\mathcal{A}_n \subseteq \Theta_{n-1} \cup \Theta_n \cup \Theta_{n+1}$, which will imply that $\mathcal{A}_n$ is finite since each $\Theta_k$ is a finite graph.

    Let $j > i$. Choose vertices $a \in p_i$ and $b \in p_j$ such that $d_Y(1,a) = d_Y(1,b) = (x_i,x_j)_1$. Note that such vertices $a,b$ exist since $(x_i,x_j)_1 \leq \min\{d_Y(1,x_i), d_Y(1,x_j)\}$. 
    
    Then, by Theorem \ref{thm: Strebel's bigon and triangle classification}, the vertices $1,a,b$ in $X$ form the geodesic triangle in $X$ as in Figure \ref{geodesic triangles and bigons}, which is composed of bigonal and triangular diagrams filled by contours. \textcolor{black}{As discussed in Section \ref{section: graphical small cancellation}, we can assume that each of the bigonal and triangular diagrams are reduced.}

    \begin{figure}[h]
        \centering
        \includegraphics[width=0.45\linewidth]{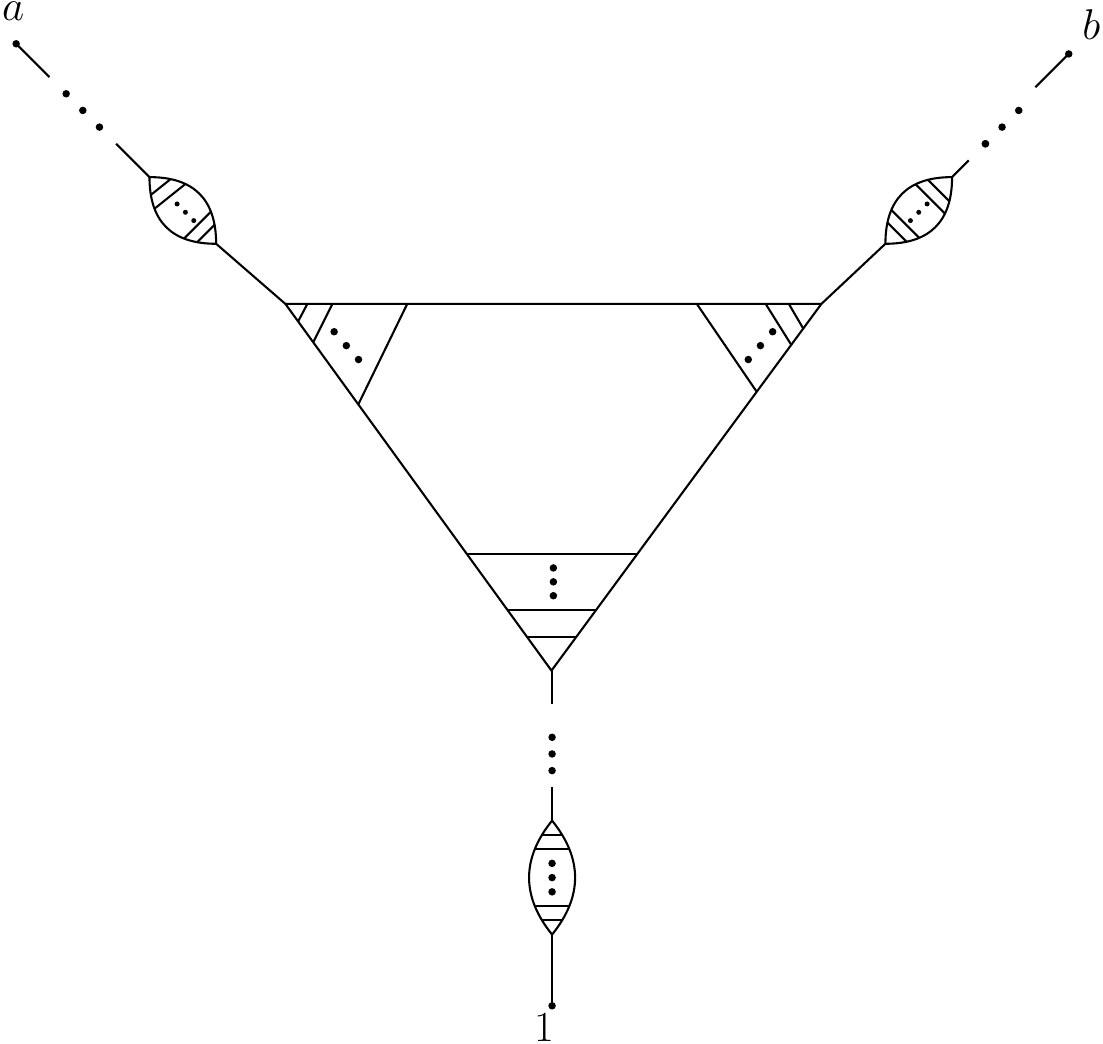}
        \caption{The simple geodesic triangle and bigons appearing in the proof of Lemma \ref{lemma: geodesic representatives for boundary}. The segments of $p_i,p_j$ between $1,a$ and $1,b$ as well as the geodesic between $a,b$ bound \textbf{black}{reduced} bigonal and triangular diagrams filled by contours. The simple geodesic triangle is illustrated as having shape II, but it may have any of the shapes of simple geodesic triangles shown in Theorem \ref{thm: Strebel's bigon and triangle classification}.}
        \label{geodesic triangles and bigons}
    \end{figure}

    Let $v \in \mathcal{A}_n$. We will show that $v$ must occur on or before the simple geodesic triangle in Figure \ref{geodesic triangles and bigons}. 
    
    Indeed, if $v$ were beyond the simple geodesic triangle, then there exists a vertex $v'$ on the geodesic connecting $a$ and $b$ such that $d_Y(v,v')\leq 1$. Since $d_Y(a,b) \leq \delta$, we have $d_Y(b, v') \leq \delta$. Thus, by the triangle inequality, $d_Y(v,b) \leq \delta + 1$. This yields $d_Y(1,b) \leq d_Y(1,v) + d_Y(v,b) \leq n + \delta + 1$, contradicting that $d_Y(1,b) = (x_i, x_j)_1 > n + \delta + 2$. 

    Therefore, $v$ is indeed before or on the simple geodesic triangle in Figure \ref{geodesic triangles and bigons}. It follows that $v$ is on $p_i$ or $v$ is on a common contour $r$ that $p_i,p_j$ both pass through. Indeed, if $v$ is before the simple geodesic triangle in Figure \ref{geodesic triangles and bigons}, then by the classification of simple geodesic bigons, we have $v$ is on a contour bounded by subsegments of $p_i$ and $p_j$, or $v$ is on $p_i$ if $p_i,p_j$ coincide at $v$. If $v$ is on the simple geodesic triangle and not on a contour bounded by $p_i$ and $p_j$, then by the classification of simple geodesic triangles (Theorem \ref{thm: Strebel's bigon and triangle classification}), we would obtain a vertex $v'$ on $[a,b]$ such that $d_Y(v,v') \leq 2$, which would yield a contradiction to $(x_i, x_j)_1 > n + \delta + 2$ as above. Therefore, we must have that $v$ is on a common contour of $p_i$ and $p_j$.

    Note that if $r$ is a contour bounded by subsegments of $p_i$ and $p_j$, then we must have
    \begin{align}\label{eq:big intersection}
        \vert r \cap p_i \vert > \Big(\frac{1}{2} - 2 \lambda\Big) \vert r \vert \geq 3\lambda \vert r \vert.
    \end{align}

    Indeed, let $s,t$ be contours adjacent to $r$ in the diagram bounded by $p_i$ and $p_j$ (or vertices if $r$ is the initial or terminal contour in the diagram, or if the diagram consists of a single contour). See Figure \ref{fig:contours u and v}. 

    \begin{figure}[h]
        \centering
        \includegraphics[width=0.5\linewidth]{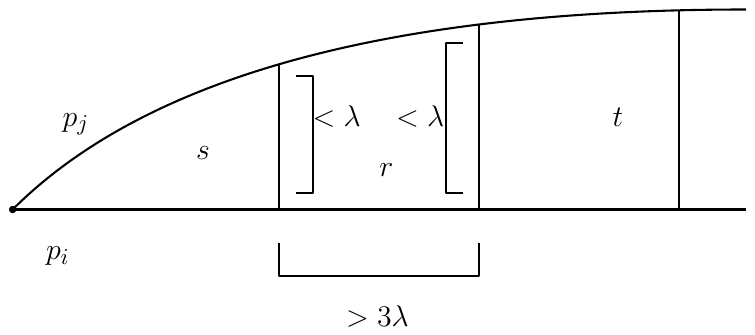}
        \caption{The contours $s,t$ adjacent to $r$ filling a diagram formed by $p$ and $q$.}
        \label{fig:contours u and v}
    \end{figure}

   \textcolor{black}{Since we assume each of bigonal and triangular diagrams in the diagram bounded by $p_i$ and $p_j$ are reduced}, we have that if $s,t$ are non-trivial, then $s,r$ and $r,t$ are contained in different relators. By the small cancellation condition, we then have $\vert r \cap s \vert < \lambda \vert r \vert $ and $\vert r \cap t \vert < \lambda \vert r \vert $. \textcolor{black}{Note that if $s$ (respectively, $t$) are vertices, i.e.\ the contour $s$ (respectively, $t$) does not exist, then $\vert r \cap s \vert = 0 < \lambda \vert r \vert$ (respectively, $\vert r \cap t \vert = 0 < \lambda \vert r \vert$), so the inequalities $\vert r \cap s \vert < \lambda \vert r \vert $ and $\vert r \cap t \vert < \lambda \vert r \vert $ still hold even if the contours $s$ or $t$ do not exist.} Since $p_j$ is a geodesic, we must have $\vert p_i \cap r \vert > (\frac12 - 2\lambda) \vert r \vert$, since if $\vert p_i \cap r \vert \leq (\frac12 - 2\lambda) \vert r \vert$, then the path along $r$ consisting of $s \cap r, p_i \cap r$ and $t \cap r$ would have length less than $\lambda \vert r \vert + (\frac12 - 2\lambda) \vert r \vert + \lambda \vert r \vert = \frac12 \vert r \vert$, hence this path would be shorter than $p_j \cap r$ and have the same endpoints as $p_j \cap r$, contradicting that $p_j \cap r$ is a geodesic (being a subpath of the geodesic path $p_j$). 

    By choice of $\lambda$, we have $\frac12 - 2 \lambda \geq  3\lambda$, hence we have $\vert p \cap r \vert >  3\lambda \vert r \vert$, as desired. 

    Now, if a vertex $v \in \mathcal{A}_n$ is on $p_i$, then $v \in p_{i,n-1} \cup p_{i,n} \cup p_{i,n+1} \subseteq \Theta_{n-1} \cup \Theta_{n} \cup \Theta_{n+1}$. Otherwise, $v$ is on a contour $r$ bounded by segments of $p_i$ and $p_j$, and contained in a relator $\Theta$. We will show in this case that $\Theta = \Theta_k$ for some $k \in \{n-1,n,n+1\}$. Suppose that $\Theta \neq \Theta_k$ for any $k \in \N$. Then the subsegment $\Theta \cap p_i$ of $p_i$ must intersect more than 3 $\Theta_k$. \textcolor{black}{Indeed, otherwise $\Theta \cap p_i$ would be covered by at most 3 $\Theta_k$, and since  $\Theta \neq \Theta_k$ for any $i$, we would obtain that $\Theta \cap \Theta_k \cap p_i$ is a piece for each $k$, so that $\vert \Theta \cap p_i \vert < 3 \lambda \girth(\Theta)$, but by above we have that $\vert p_i \cap r \vert > 3 \lambda \vert r \vert \geq 3 \lambda \girth(\Theta)$, and since $r \subseteq \Theta$, this contradicts that $\vert p_i \cap \Theta \vert < 3 \lambda \mathrm{girth}(\Theta)$ obtained from the assumption above that $\Theta \cap p_i$ intersects more than 3 $\Theta_k$. Thus, $\Theta \cap p_i$ must intersect more than 3 $\Theta_k$ if $\Theta \neq \Theta_k$ for any $k$. Then there exists $k_1, k_2 \in \N$ with $k_2 - k_1 \geq 3$ such that $ \Theta \cap p_i \cap \Theta_{k_1} \neq \emptyset$ and $ \Theta \cap p_i \cap \Theta_{k_2} \neq \emptyset$. Take vertices $x \in \Theta \cap p_i \cap \Theta_{k_1}$ and $y \in \Theta \cap p_i \cap \Theta_{k_2}$. Then $d_Y(x,y) \leq 1$ since $x,y$ are in the same relator $\Theta$, but on the other hand, by Proposition \ref{Geod decomp}, we have that $d_Y(x,y) \geq 2$, a contradiction. Therefore, we must have $\Theta = \Theta_k$ for some $k$. Since $\Theta$ contains a vertex $v$ with $d_Y(1,v) = n$, we must have $k \in \{n-1,n,n+1\}$. }
    
    This proves that $\mathcal{A}_n \subseteq \Theta_{n-1} \cup \Theta_n \cup \Theta_{n+1}$, and hence that $\mathcal{A}_n$ is finite. 

    We now construct the sequence $(a_n)_n \subset G$ by induction. 

    Since $\mathcal{A}_1$ is finite, there is a subsequence $(p_{n_k})_k$ of $(p_n)_n$ and a vertex $a_1 \in \mathcal{A}_1$ such that $p_{n_k}$ all pass through $a_1$. Considering the sequence of geodesic segments $(p_{n_k})_k$, repeating the same argument above yields a vertex $a_2 \in \mathcal{A}_2$ such that infinitely many $p_{n_k}$ pass through $a_2$. Continuing inductively in this manner, a diagonalization argument yields a sequence $(a_n)_n$ with $a_n \in \mathcal{A}_n$ for all $n$ and a subsequence $(p_{j_n})_n$ of $(p_n)_n$ such that each $p_{j_n}$ passes through $a_m$ for all $m \leq n$. 
    
    Note that there are only finitely many geodesic segments between any two vertices $v,w$ of $X$, since letting $n = d_Y(v,w)$ and fixing a geodesic sequence $(\Theta_i)_{i=1}^n$ with $v \in \Theta_1$ and $w \in \Theta_n$ (c.f.\ Proposition \ref{Geod decomp}), by \cite[Remark 3.7]{Gruber_Sisto} we have that any geodesic $\gamma$ from $v$ to $w$ in $X$ is contained in $\cup_{i=1}^n \Theta_i$, which is a finite subgraph of $X$ (since each $\Theta_i$ is finite). Therefore, the subsequence $(p_{j_n})_n$ has a subsequence converging to a geodesic ray $p$ in $X$ with $a_n \in p$ for all $n$, since for each $m$ there are infinitely many $p_{j_n}$ with a common subsegment from 1 to $a_m$.
    
    We have $\lim_{n \to \infty}d_Y(1,a_n) = \lim_{n \to \infty}n = \infty$ by $a_n \in \mathcal{A}_n$. Hence, by Lemma \ref{lem:geod ray converging in Y}, it remains to show that $a_n \to \xi$ in $Y$. Recall that for each $n$, we have $a_n \in p_{j_n}$ where $p_{j_n}$ is a geodesic segment from $1$ to $x_{j_n}$ with $d_Y(1,x_{j_n}) \geq n$. Decomposing $p_{j_n}$ into a geodesic sequence of edges in $Y$ as above, we obtain $d_Y(a_n, x_{j_n}) \leq d_Y(1, x_{j_n}) - n + 1$. This and $d_Y(1,a_n)=n$ yield that $(a_n, x_{j_n})_1 \geq n-1$. 
    Hence $(a_n, x_{j_n})_1 \to \infty$ as $n \to \infty$, and thus $(a_n)_n \sim (x_{j_n})_n$. Since $(x_{j_n})_n$ is a subsequence of $(x_j)_j$, we have that $(a_n)_n \sim (x_j)_j$. Thus, since $x_j \to \xi$ in $Y$, we have that $a_n \to \xi$ in $Y$. Therefore, $p \in G(\xi)$.
\end{proof}

\begin{cor}
    \label{cor: Geodesic and seq boundaries same}
    The geodesic boundary of $Y$ (i.e.\ the set of all geodesic rays in $Y$ based at 1 modulo finite Hausdorff distance with respect to $d_Y$) coincides with the sequential boundary $\partial Y$ as a topological space.
\end{cor}

\begin{proof}
Let $\partial_g Y$ denote the geodesic boundary of $Y$. We have a natural map $\iota: \partial_g Y \to \partial Y$, since each geodesic ray $(a_n)_{n \in \N}$ in $Y$ converges to infinity in $Y$, and hence defines a point in $\partial Y$. By definition of the topology on $\partial Y$ in terms of the Gromov product and the topology of $\partial_g Y$ in terms of geodesics fellow travelling for longer distances (see, for instance \cite[Chapter III.H.3]{BH99}), this map $\iota$ is a homeomorphism onto its image. Lemma \ref{lemma: geodesic representatives for boundary} shows that $\iota$ is surjective. Thus, $\iota$ is a homeomorphism. 
\end{proof}

{\color{black}

Lemma \ref{lem:big relators form geod} below is used in the proof of Lemma \ref{lem: form of geodesics in bundle}.

    \begin{lem}\label{lem:big relators form geod}
        Let $x,y \in G$ and let $p$ be a geodesic path in $X$ from $x$ to $y$. Let $x=a_0,a_1,\ldots,a_k=y$ be a sequence of vertices on $p$ with $k\in \N$. Suppose that for every $i \in \{1,\ldots,k\}$, there exists a relator $\Theta_i$ in $X$ such that $p_{[a_{i-1},a_i]} \subset \Theta_i$, $|p_{[a_{i-1},a_i]}| \ge 3\lambda girth(\Theta_i)$, and $\Theta_i \neq \Theta_{i+1}$. Then, the sequence $(a_i)_{i=0}^k$ is geodesic in $Y$.
    \end{lem}

    \begin{proof}
        Suppose for contradiction that there exist $i,j \in \{1,\ldots,k\}$ such that $i < j$ and $\Theta_i = \Theta_j$. Note $i+1<j$ by $\Theta_i \neq \Theta_{i+1}$. By $\{a_{i-1},a_j\} \subset \Theta_i$ and Lemma \ref{geodesic subpath}, we have $p_{[a_{i-1},a_j]} \subset \Theta_i$. Hence, $p_{[a_i,a_{i+1}]} \subset \Theta_i \cap \Theta_{i+1}$. By this and $|p_{[a_i, a_{i+1}]}| \ge 3\lambda girth(\Theta_{i+1})$, we have $\Theta_i = \Theta_{i+1}$ by $C'(\lambda)$-condition. This contradicts $\Theta_i \neq \Theta_{i+1}$. Thus, $\Theta_i$'s are all distinct.

    Set $\ell=d_Y(x,y)$. By Proposition \ref{Geod decomp}, we can take a decomposition $p = \gamma_{1} \cdots \gamma_{\ell}$, where each $\gamma_{i}$ is either a path in some relator $\Delta_i$ in $X$ or an edge in $X$ not occurring on any relator. In the latter case, we define $\Delta_i$ by $\Delta_i = \gamma_i$. By $\ell=d_Y(x,y)$, we have $\ell \le k$.

    Suppose for contradiction that there exists $n \in \{1,\ldots,k\}$ such that $\Theta_n \notin \{\Delta_m \mid m \in \{1,\ldots,\ell\}\}$. There exist $i,j \in \{1,\ldots,\ell\}$ with $i \le j$ such that $a_{n-1}\in \gamma_i$ and $a_n \in \gamma_j$. If $j>i+2$, then we have $\gamma_{i+1}\gamma_{i+2} \subset p_{[a_{n-1},a_n]}$, which contradicts $d_Y(\gamma_{(i+1)-},\gamma_{(i+2)+})=2$, where $\gamma_{(i+1)-}$ is the initial vertex of $\gamma_{i+1}$ and $\gamma_{(i+2)+}$ is the terminal vertex of $\gamma_{i+2}$. On the other hand, if $j \le i+2$, then by $C'(\lambda)$-condition, we have $|p_{[a_{n-1},a_n]}| \le \sum_{s=i}^{i+2}|\Theta_n \cap\Delta_s \cap p| < 3\lambda girth(\Theta_n)$, which contradicts $|p_{[a_{n-1},a_n]}| \ge 3\lambda girth(\Theta_n)$.

    Thus, $\{\Theta_n\}_{n=1}^k \subset \{\Delta_m\}_{m=1}^\ell$. This implies $k\le \ell$, hence $k =\ell =d_Y(x,y)$. Hence, the sequence $(a_i)_{i=0}^k$ is geodesic in $Y$.
    \end{proof}

    \begin{lem}
    \label{lem: form of geodesics in bundle}
    Let $\xi \in \partial Y$ and $p,q \in G(\xi)$. The following hold.
    \begin{itemize}
    \item[(1)]
    There exists a sequence of vertices $(v_i)_{i \in \N}$ on $p$ and $(w_i)_{i \in \N}$ on $q$ such that for any $i \geq 1$, either (a) or (b) below holds (see Figure \ref{fig:Geodesics in bundle}).
    
    \begin{itemize}
        \item[(a)] $p_{[v_{i-1}, v_i]} = q_{[w_{i-1}, w_i]}$.
        \item[(b)] There exist a relator $\Theta$ in $X$ and paths $s$ from $v_i$ to $w_i$ and $t$ from $w_{i-1}$ to $v_{i-1}$ in $X$ such that $p_{[v_{i-1}, v_i]} s q_{[w_{i-1}, w_i]}^{-1}t$ is a contour in $\Theta$ and we have $\min\{|p_{[v_{i-1}, v_i]}|, |q_{[w_{i-1}, w_i]}|\} \ge \big(\frac{1}{2}-2\lambda\big)girth(\Theta)$.
    \end{itemize}

    %we have that $p$ and $q$ form a sequence of diagrams filled by contours as shown in Figure \ref{fig:Geodesics in bundle}. %In particular, for any $p,q \in G(\xi)$ and for any $n \in \N$, we have that $d_Y(p(n), q(n)) \leq 1$. 
    \item[(2)] 
    Let $(r_n)_{n \in \N}$ be a sequence of subpaths of $p$ as in Lemma \ref{lem:geod ray converging in Y} (3), and $\Theta_n \supset r_n$ is either a relator or $\Theta_n = r_n$ if $r_n$ is not contained in any relator, then $q \subset \cup_{n \in \N} \Theta_n$ and $q \cap \Theta_n \neq \emptyset$ for each $n$, intersecting each $\Theta_n$ either along $p$ or along a contour shared with $p$. Moreover, for any vertex $v \in q$ with $d_Y(1,v)=N \in \N$, we have $v \in \Theta_{N-1}\cup\Theta_N\cup\Theta_{N+1}$.

    \begin{figure}[h]
        \centering
        \includegraphics[width=0.5\linewidth]{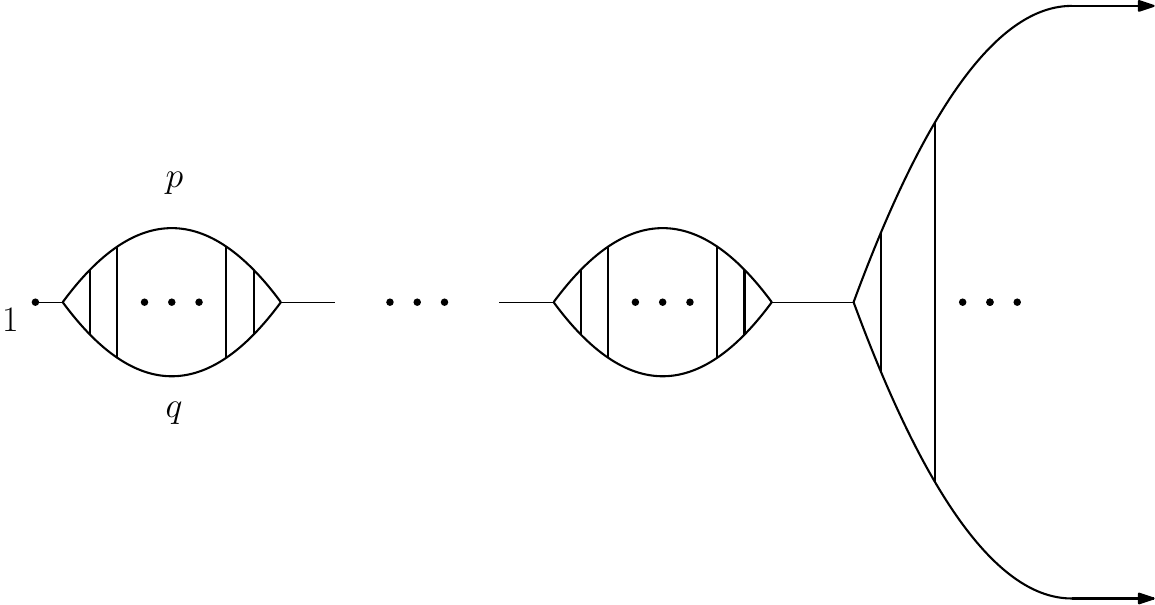}
        \caption{The form of any two geodesic rays $p,q$ in $G(\xi)$.}
        \label{fig:Geodesics in bundle}
    \end{figure}        
    \end{itemize}    
\end{lem}
}
\begin{proof}
    (1)
    We consider the following cases. 

    First, suppose that there are infinitely many $i \in \N$ such that $p(i) = q(i)$. If $p=q$, then simply set $v_i = p(i) = q(i) = w_i$ for all $i \in \N$. Otherwise, there exist sequences of natural numbers $(s_n)_{n=1}^N, (t_n)_{n=1}^N$ (with $N$ possibly infinite) with $s_n < t_n \le s_{n+1}$ for all $n$ {\color{black}(i.e. $0 \le s_1 < t_1 \le s_2 <t_2 \le \cdots$)} such that:
    
    \begin{itemize}
        \item $p(i) = q(i)$ for each $0 \leq i \leq s_1$, and $t_n \leq i \leq s_{n+1}$, and
        \item $p(i) \neq q(i)$ for each $s_n < i < t_n$.
    \end{itemize}

     For each $n$, we then have that $p_{[p(s_n), p(t_n)]}$ and $q_{[q(s_n), q(t_n)]}$ form simple geodesic bigons $B_n$ in $X$, and the segments $p_{[1, p(s_1)]}, q_{[1,q(s_1)]}$ and $p_{[p(t_n), p(s_{n+1})]}, q_{[q(t_n), q(s_{n+1})]}$ coincide. 
    
    By Theorem \ref{thm: Strebel's bigon and triangle classification}, each of the simple geodesic bigons $B_n$ bounds a diagram of shape $I_1$, filled by a sequence of contours $(r_{i,n})_{i=1}^{k_n}$. This yields a sequence of vertices $(v_{i,n})_{i=0}^{m_n}$ on $p$ and $(w_{i,n})_{i=0}^{m_n}$ on $q$ as well as a sequence of paths $(\alpha_{i,n})_{i=1}^{m_n}$ from $v_{i,n}$ to $w_{i,n}$ and $(\beta_{i,n})_{i=1}^{m_n}$ from $w_{i-1,n}$ to $v_{i-1,n}$ such that for each $i = 1,\ldots, k_n$ we have $r_{i,n} = p_{[v_{i-1,n}, v_{i,n}]} \alpha_{i,n} q_{[w_{i-1, n}, w_{i,n}]}^{-1} \beta_{i,n}$. 

    We then define the sequences $(v_i)_{i \in \N}$ and $(w_i)_{i \in \N}$ by concatenating the sequences of vertices on the segments where $p,q$ coincide and on the bigons $B_n$ (below, $\cdot$ denotes concatenation of sequences):
    
    $$(v_i)_{i \in \N} = (p(i))_{i=0}^{s_1} \cdot (v_{i,1})_{i=1}^{t_1} \cdot ((p(i))_{i=t_1+1}^{s_2}) \cdots \text{, and } (w_i)_{i \in \N}= (q(i))_{i=0}^{s_1} \cdot (w_{i,1})_{i=1}^{t_1} \cdot ((q(i))_{i=t_1+1}^{s_2}) \cdots $$
    
    If $N$ is finite, then the last terms in the above concatenations are $(p(i))_{i=t_N+1}^{\infty}$ and $(q(i))_{i=t_N +1}^{\infty}$.  

    Now suppose that there are only finitely many $i \in \N$ such that $p(i) = q(i)$. Then there exist sequences of natural numbers $(s_n)_{n=1}^{N+1}, (t_n)_{n=1}^N$ (with $N \in \N$ now finite) with $s_n < t_n \le s_{n+1}$ for all $n$ {\color{black}(i.e. $0 \le s_1 < t_1 \le s_2 <t_2 \le \cdots$)} such that:
    
    \begin{itemize}
        \item $p(i) = q(i)$ for each $0 \leq i \leq s_1$, and $t_n \leq i \leq s_{n+1}$, 
        \item $p(i) \neq q(i)$ for each $s_n < i < t_n$, and 
        \item $p(i) \neq q(i)$ for all $i > s_{N+1}$
    \end{itemize}
    
    Arguing as in the first case above, there exist sequences of vertices $(v_i)_{i=1}^M, (w_i)_{i=1}^M$ on $p_{[1, p(s_{N+1})]}$ and $q_{[1, q(s_{N+1})]}$ with the desired properties. It remains to show that there exist desired sequences of vertices on $p_{[p(s_{N+1}), \infty) }$ and $q_{[q(s_{N+1}), \infty)}$. {\color{black} Note that $p_{[p(s_{N+1}), \infty) } \cap q_{[q(s_{N+1}), \infty)} = \{p(s_{N+1})\} \,(=\{q(s_{N+1})\})$ since $p$ and $q$ are geodesic and we have $\forall\, i > s_{N+1},\, p(i) \neq q(i)$.

    Set $o = p(s_{N+1})$. By $p, q \in G(\xi)$, we have $\lim_{\ell,m \to \infty}(p(\ell),q(m))_o = \infty$, where the Gromov product is in $Y$. Hence, by Proposition \ref{Geod decomp}, for each $n \in \N$, there exist vertices $v_n \in p$ and $w_n$ such that $d_Y(v_n,w_n) \le \delta$ and $\min\{d_Y(o,v_n), d_Y(o,w_n)\} \ge n+\delta+2$. By applying Theorem \ref{thm: Strebel's bigon and triangle classification} to a geodesic triangle in $X$ formed by $p_{[o,v_n]}$, $q_{[o,w_n]}$, and some geodesic in $X$ from $v_n$ to $w_n$ in the same way as the proof of Lemma \ref{lemma: geodesic representatives for boundary}, we can see that there exists a sequence of vertices $(a_{n,i})_{i = 0}^{n}$ on $p_{[o,v_n]}$ and $(b_{n,i})_{i = 0}^{n}$ on $q_{[o,w_n]}$, where $a_{n,0}=b_{n,0}=o$, that satisfy the following property:
    \begin{itemize}
        \item[($\ast$)]for any $i \in \{1,\ldots,n\}$, there exist a relator $\Theta_{n,i}$ and a path $s_{i,n}$ in $X$ from $a_{n,i}$ to $b_{n,i}$ such that the loop $p_{[a_{n,i-1},a_{n,i}]}s_{n,i}(q_{[b_{n,i-1},b_{n,i}]})^{-1}s_{n,i-1}^{-1}$ is a contour in $\Theta_{n,i}$ and we have $\Theta_{n,i-1} \neq \Theta_{n,i}$.
    \end{itemize}
    
    \textcolor{black}{Since $\frac12 - 2\lambda\ge 3\lambda$, using the same argument illustrated in Figure \ref{fig:contours u and v} in the proof of Lemma \ref{lemma: geodesic representatives for boundary} and applying Lemma \ref{lem:big relators form geod}, the sequences $(a_{n,i})_{i = 0}^{n}$ and $(b_{n,i})_{i = 0}^{n}$ are a geodesic path in $Y$.} Hence, for any $i,m,n \in \N$ with $m,n \ge i+1$, we have $a_{m,i} \in p_{[o,a_{n,i+1}]}$ and $b_{m,i} \in q_{[o,b_{n,i+1}]}$. Indeed, $a_{m,i} \notin p_{[o,a_{n,i+1}]}$ implies $d_Y(o,a_{n,i+1}) \le i$, which contradicts $d_Y(o,a_{n,i+1}) = i+1$ (and the same argument holds for $q$). 
    
    Thus, $\#\{a_{n,i} \mid n \ge i \} < \infty$ for any $i \in \N$. By taking subsequences and diagonal argument, there exist a sequence of vertices $(a_i)_{i \in \N}$ on $p$ and $(b_i)_{i \in \N}$ on $q$ and a subsequence $(n_j)_{j \in \N} \subset \N$ such that for any $i,j \in \N$ with $i \le j$, we have $a_{n_j,\,i} = a_i$ and $b_{n_j,\,i} = b_i$. Hence, the sequences $(a_i)_{i \in \N}$ and $(b_i)_{i \in \N}$ satisfy condition (b) by the property ($\ast$).

    (2)
    By Lemma \ref{lem: form of geodesics in bundle} (1) and $\frac12 - 2\lambda\ge 3\lambda$, it's enough to show that for any relator $\Theta$ in $X$ satisfying $|\Theta \cap p| \ge 3 \lambda girth(\Theta)$, there exists $i \in \N$ such that $\Theta = \Theta_i$. Suppose $\Theta \notin \{\Theta_n \mid n \in \N\}$ for contradiction. Since the decomposition of $p$ into $r_n$'s provides a geodesic ray in $Y$, there exist $i \in \N$ such that $\Theta \cap p \subset \Theta_i\cup\Theta_{i+1}\cup\Theta_{i+2}$. Since we have $|\Theta \cap\Theta_k \cap p| < \lambda girth(\Theta)$ for any $k\in \N$ by $\Theta \neq\Theta_k$ and $C'(\lambda)$-condition, this implies 
    $$|\Theta \cap p| \le \sum_{k=i}^{i+2}|\Theta \cap\Theta_k \cap p| < 3\lambda girth(\Theta),$$
    which contradicts $|\Theta \cap p| \ge 3 \lambda girth(\Theta)$.

    To show the ``moreover" part, let $v \in q$ satisfy $d_Y(1,v) = N \in \N$. For each $n \in \N$, let $r_{n-}$ and $r_{n+}$ be the initial vertex and the terminal vertex of the path $r_n$ respectively. There exists $i \in \N$ such that $v \in \Theta_i$. If $i < N-1$, then we have $d_Y(1,v) \le d_Y(1,r_{i-}) + d_Y(r_{i-},v)\le i-1+1<N$, which contradicts $d_Y(1,v)=N$. On the other hand, if $i>N+1$, then we have $d_Y(1,r_{i+}) \le d_Y(1,v) + d_Y(v,r_{i+})\le N+1<i$, which contradicts $d_Y(1,r_{i+}) = i$. Thus, $N-1 \le i \le N+1$.
    }
\end{proof}

From now on, when we refer to $\partial Y$, we will mean the geodesic boundary of $Y$.

%(in fact, the argument in the ``moreover" statement in Lemma \ref{1} shows that any two geodesic rays in $X$ representing the same boundary point in $\partial Y$ pass through the same sequence of contours). 

We next establish that there exists a \emph{lexicographically least} geodesic ray in $G(\xi)$. Fix an arbitrary well-order on $S$. Using this well-order, we obtain a lexicographic order on $S^{\N}$, hence on geodesic rays in $X$ (via the labels of the geodesic rays in $S^{\N}$). We will deduce the existence of a lexicographically least geodesic in $G(\xi)$ from the compactness of $G(\xi)$ in $S^{\N}$. %Recall that with the product topology, $S^{\N}$ is a compact metrizable space.

{\color{black} \begin{cor}\label{cor:G(xi) is compact}
    For any $\xi \in \partial Y$, the following hold.
    \begin{itemize}
        \item[(1)]
        The subgraph in $X$ induced by $\bigcup_{p \in G(\xi)}p$ is locally finite.
        \item[(2)] 
        $G(\xi)$ is compact as a subspace of $S^{\N}$.
    \end{itemize}
\end{cor}

\begin{proof}
    (1)
    Fix a geodesic ray $p \in G(\xi)$. Let $(r_n)_{n \in \N}$ be a sequence of subpaths of $p$ as in Lemma \ref{lem:geod ray converging in Y} (3), and $\Theta_n \supset r_n$ is either a relator or $\Theta_n = r_n$ if $r_n$ is not contained in any relator. For each $n \in \N$, let $r_{n-}$ and $r_{n+}$ be the initial vertex and the terminal vertex of the path $r_n$ respectively. By Lemma \ref{lem: form of geodesics in bundle} (2), every vertex in $\bigcup_{q \in G(\xi)}q$ is contained in $\bigcup_{n \in \N} \Theta_n$. Hence, it's enough to show that the subgraph in $X$ induced by $\bigcup_{n \in \N} \Theta_n$ is locally finite.

    Let $v \in \Theta_n$ and $w \in \Theta_m$ be vertices such that $d_X(v,w)=1$, where $n,m \in \N$. If $m < n-2$, then we have
    \begin{align*}
        d_Y(1,r_{n+}) \le d_Y(1,r_{m-}) + d_Y(r_{m-},w) + d_Y(w,v) + d_Y(v,r_{n+})\le (m-1)+1+1+1<n,
    \end{align*}
    which contradicts $d_Y(1,r_{n+})=n$. Hence, $m \ge n-2$. In the same way, we also get $n \ge m-2$. Thus, we have $\{w' \in \bigcup_{k \in \N} \Theta_k \mid d_X(v,w')=1\} \subset \bigcup_{k = n-2}^{n+2} \Theta_k$. This implies local finiteness of the induced subgraph of $\bigcup_{k \in \N} \Theta_k$.

    (2)
    %Since $S$ is finite, $S^{\N}$ is compact, so it suffices to show that $G(\xi)$ is closed in $S^{\N}$. 
     Let $(p_n)_{n \in \N}$ be a sequence of geodesic rays in $G(\xi)$. Denote $p := p_1$. Fix a decomposition $(r_i)_{i \in \N}$ of $p$ into subsegments as in Lemma \ref{lem:geod ray converging in Y} (3). For each $i$, let $\Theta_i$ be a relator containing $r_i$ if such a relator exists or $r_i$ if $r_i$ is not contained in any relator. By Lemma \ref{lem: form of geodesics in bundle} (2), for each $n \in \N$, the set $\mathcal{A}_n := \{v \in \bigcup G(\xi): d_Y(1,v) = n\}$ is contained in $\Theta_{n-1} \cup \Theta_n \cup \Theta_{n+1}$, hence is finite. 

     By $\forall\,n\in\N,\,\#\mathcal{A}_n<\infty$ and Corollary \ref{cor:G(xi) is compact} (1), there exists a subsequence $(p_{n_k})_k$ of $(p_n)_n$ (taken by diagonal argument) which converges to a geodesic ray $q$ in $X$ that passes through a sequence of vertices $(v_k)_{k \in \N}$ with $d_Y(1,v_k) = k$, $d_Y(p,v_k) \leq 1$ for all $k \in \N$. {\color{black} Since $p \in G(\xi)$ and $d_Y(p, v_k) \leq 1$ for all $k$, this implies that $q \in G(\xi)$ by Lemma \ref{lem:geod ray converging in Y}.} We conclude that $G(\xi)$ is compact. 
     
\end{proof}
}

The following lemma is standard, but we record its proof for completeness.

\begin{lem}
    \label{3}
    For each non-empty {\color{black}closed} $K \subseteq S^{\N}$, there exists a lexicographically least element of $K$.
\end{lem}

\begin{proof}
    {\color{black}For each $n\in \N$, we define the element $s_n\in S$ and the subset $K_n$ of $K$ inductively as follows:
    \begin{align*}
    s_1&=\min \{ w_1 \in (S,\le) \mid \exists\, w\in K,\, w=(w_1,w_2,\ldots) \}, \\
    K_1&=\{ w\in K \mid w=(s_1,w_2,\ldots) \}, \\
    s_{n+1}&=\min \{ w_{n+1} \in (S,\le) \mid \exists\, w\in K_n,\, w=(s_1,\ldots,s_n,w_{n+1},\ldots) \}, \\
    K_{n+1}&=\{ w\in K_n \mid w=(s_1,\cdots,s_n,s_{n+1},\ldots) \}.
    \end{align*}
    Note that each $K_n$ is nonempty since $K$ is nonempty and $\le$ is a well-order on $S$. We define the element $s\in S^\N$ by $s=(s_1,s_{2},s_{3},\ldots)$ and take an element $t_n \in K_n$ for each $n\in\N$. Since $(t_n)_{n=1}^\infty$ converges to $s$ in $S^\N$ and $K$ is closed, we have $s \in K$. By $s \in\bigcap_{n=1}^\infty K_n$, the element $s$ is the lexicographically least in $K$.}
\end{proof}

\begin{cor}
    \label{4}
     For each $\xi \in \partial Y$, there exists a lexicographically least geodesic ray in $G(\xi)$. 
\end{cor}

\begin{proof}
    By Corollary \ref{cor:G(xi) is compact} (2), $G(\xi)$ is compact in $S^{\N}$, and hence by Lemma \ref{3}, there exists a lexicographically least geodesic ray in $G(\xi)$.
\end{proof}

\begin{defn}
    For each $\xi \in \partial Y$, using Corollary \ref{4}, put $\sigma_{\xi}$ to be the label of the lexicographically least geodesic ray in $G(\xi)$. %We have that $\sigma_{\xi}$ is well-defined by Corollary \ref{4}. 
    We then define a map $\Phi : \partial Y \to S^{\N}$ via $\xi \mapsto \sigma_{\xi}$. 
\end{defn}

\begin{lem}
    \label{Borel inj}
    The map $\Phi : \partial Y \to S^{\N}$ is a Borel injection.
\end{lem}

To prove \ref{Borel inj}, we closely follow the arguments of \cite[Proposition 3.3]{NaryshVacc}. Recall that we identify geodesic rays in $G(\xi)$ with their labels in $S^{\N}$.

\begin{lem}
    \label{closed}
    The set $A = \{(\xi, p) \in \partial Y \times S^{\N}: p \in G(\xi)\}$ is closed in $\partial Y \times S^{\N}$.
\end{lem}

\begin{proof}

Let $(\xi_n, p_n)_n \subset A$ converge to $(\xi, \gamma) $ in $\partial Y \times S^{\N}$. 

    We show that there exists a subsequence $(p_{n_k})_k$ of $(p_n)_n$ and a geodesic ray $q \in G(\xi)$ such that $p_{n_k} \to q$.

    Let $p \in G(\xi)$ be arbitrary. For each $n \in \N$, let $(a_m^n)_m \subset p_n$ be a sequence of vertices on $p_n$ as in Lemma \ref{lem:geod ray converging in Y} (3). Similarly, for each $n$, let $(b_n)_n \subset p$ be a sequence of vertices on $p$ and $(r_i)_i$ a sequence of subpaths as in Lemma \ref{lem:geod ray converging in Y} (3). Denote $\Theta_i$ a fixed relator containing \textcolor{black}{$r_i$ or the subpath $r_i$ itself} if $r_i$ is not contained in a relator. 
    
    Fix $n \in \N$. Since $\xi_i \to \xi$, there exists $i \in \N$ such that for all $j > i$, we have $(a_i^j, b_i)_1 > n + \delta + 2$. 

    Letting $\mathcal{A}_n := \{v \in p_j : j > i \text{ and } d_Y(1,v) = n\}$ and arguing as in the proof of Lemma \ref{lemma: geodesic representatives for boundary}, we have that $\mathcal{A}_n \subseteq \Theta_{n-1} \cup \Theta_n \cup \Theta_{n+1}$, hence $\mathcal{A}_n$ is finite. A compactness argument as in the proof of Lemma \ref{lemma: geodesic representatives for boundary} yields that there exists a subsequence of $(p_i)_i$ which converges in $S^{\N}$ to a geodesic ray $q$ in $X$ containing a sequence of vertices $(v_n)_n$ such that for each $n \in \N$, $d_Y(p,v_n) \leq 1$ and $d_Y(1,v_n) = n$ for all $n \in \N$. Therefore, $q \in G(\xi)$. 
    %d_Y(v_m,v_n) = \vert n - m \vert$ for all $n, m \in \N$. 

    Since $p_n \to \gamma$ and a subsequence of $(p_n)_n$ converges to $q$, we have $\gamma = q$. Hence, $\gamma \in G(\xi)$. 

    We conclude that $A$ is closed. 
    
\end{proof}

\begin{proof}[Proof of Lemma \ref{Borel inj}]

    First, recall that $G(\xi)$ is a compact subset of $S^{\N}$ by Corollary \ref{cor:G(xi) is compact} (2). 

    Let $\mathcal{K}$ denote the space of compact subsets of $S^{\N}$ with the Vietoris topology (see \cite[$\mathsection$I.4.4]{Kech95}). Define a map $\psi: \partial Y \to \mathcal{K}$ by $\xi \mapsto G(\xi)$. We show that $\psi$ is Borel. 

    By Lemma \ref{closed}, we have that the set $$A = \{(\xi, p) \in \partial Y \times S^{\N}: p \in G(\xi)\}$$ is closed in $\partial Y \times S^{\N}$.

    Furthermore,  for each $\xi \in \partial Y$, the section: 

    $$A_{\xi} := \{p \in S^{\N} : p \in G(\xi)\}$$

    is equal to $G(\xi)$, which is compact by Corollary \ref{cor:G(xi) is compact} (2). Therefore, by \cite[Theorem 28.8]{Kech95}, the map $\psi: \xi \mapsto G(\xi) = A_{\xi}$ is Borel.

    Now define the map $\rho : \mathcal{K} \to S^{\N}$ defined by $K \mapsto \min_{\leq_{\mathrm{lex}}}(K)$ if $K \neq \emptyset$, where $\min_{\leq_{\mathrm{lex}}}(K)$ denotes the lexicographically least element of $K$ (which exists by Lemma \ref{3}). If $K = \emptyset$, then we define $\rho(K) = (s_M, s_M,\ldots)$, where $s_M$ is the largest element of $S$. By the proof of \cite[Proposition 3.3]{NaryshVacc}, we have that $\rho$ is continuous, hence Borel.

    Therefore, the map $\Phi = \rho \circ \psi$ is Borel. We have that $\Phi$ is injective since geodesic rays from 1 are uniquely determined by their labels.  
    
\end{proof}

\subsection{Proof of the Main Theorem}%\DO{there is no Theorem 1 now} \textcolor{black}{Chris: fixed. I also added a paragraph below to introduce the subsection.}

In this section, we define the notion of an ``extremely fine'' graph (Definition \ref{def: extremely fine})\hypersetup{linkcolor = black} and we show that as a consequence of the underlying graph being extremely fine, the action of the graphical small cancellation group on the boundary of its coned-off Cayley graph induces a hyperfinite orbit equivalence relation, proving the \hyperlink{Main Theorem}{Main Theorem}.
\hypersetup{linkcolor = red}
\begin{defn}
    Denoting $E_t$ the tail equivalence relation on $S^{\N}$ (see Definition \ref{def:tail equivalence}), define the relation $R_t' = \Phi^{-1}(E_t)$ on $\partial Y$ by $\xi R_t' \eta \iff \sigma_{\xi} E_t \sigma_{\eta}$. Since $E_t$ is hyperfinite by Proposition \ref{prop:DJK} and since $\Phi$ is a Borel injection, it follows that $R_t'$ is hyperfinite. By \cite[Proposition 1.3 (i)]{JKL02}, it follows that $R_t := E_G \cap R_t'$ is hyperfinite. 
\end{defn}    

In Proposition \ref{prop: finite index} below, we assume that the graph $\Gamma := \coprod_n \Gamma_n$ satisfies the following property, which we call \emph{extreme fineness}, a strengthening of the notion of \emph{fineness} of a graph defined by Bowditch \cite{BowditchRelHyp} (see Figure \ref{fig:bounded page}).   

\begin{defn}
    \label{def: extremely fine}
    A graph $\Gamma$ is \textbf{extremely fine} if there exists $K \in \N$ such that for every edge $e$ in $\Gamma$, there are at most $K$ simple closed paths $\gamma$ containing $e$.
\end{defn}

\begin{figure}[h]
    \centering
    \includegraphics[width=0.2\linewidth]{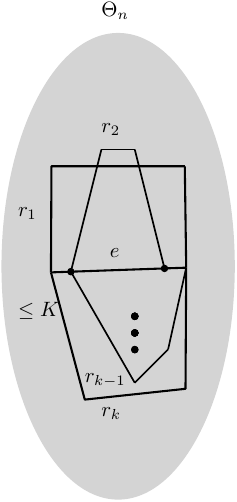}
    \caption{Extreme fineness of a graph. There is a uniform constant $K$ such that there are at most $K$ simple closed paths sharing the same edge $e$.}
    \label{fig:bounded page}
\end{figure}

Extreme fineness implies that for each geodesic path $p$ in a relator $\Theta$, there are at most $K$ contours in $\Theta$ containing any given edge of $p$. Note that every classical small cancellation presentation has an extremely fine underlying graph (which is a disjoint union of simple closed paths), since each edge is contained in a unique simple closed path (a single contour).

\begin{prop}
    \label{prop: finite index}
    With the notation as above, let $G = \langle S \vert \Gamma_n : n \in \N \rangle$ with $S$ countable be a $C'(\lambda)$ graphical small cancellation presentation with $\lambda \leq \frac1{10}$ and the graph $\Gamma = \coprod_n \Gamma_n$ being extremely fine (Definition \ref{def: extremely fine}). Then there exists $K > 0$ such that each $E_G$-class in $\partial Y$ contains at most $K$ $R_t$-classes. 
\end{prop}

Before we begin the proof of Proposition \ref{prop: finite index}, we will need the following elementary lemma. For the proof, see for instance \cite[Lemma 3.1]{MS20}.

\begin{lem}
    \label{lem: moving geodesic basepoint}
    Let $\Theta$ be a connected graph and let $x,y$ be vertices in $\Theta$. For every geodesic ray $\gamma$ in $\Theta$ based at $x$, there exists a geodesic ray $\lambda$ based at $y$ which eventually coincides with $\gamma$. 
\end{lem}

\begin{proof}[Proof of Proposition \ref{prop: finite index}]
    
    Let $K_0$ be a constant witnessing extreme fineness of $\Gamma$ i.e.\ such that in each $\Gamma_n$, there are at most $K_0$ contours sharing a common edge. Put $K = (1+K_0)^2 + 1$. Suppose for contradiction that there exist $\xi_0,\xi_1,\ldots, \xi_K \in \partial Y$ that are in the same $E_G$-class but are pairwise $R_t$-inequivalent. For each $i = 1,\ldots,K$, let $g_i \in G$ be such that $g_i \xi_i = \xi_0$, and let $\alpha_i \in G(\xi_i)$ be the geodesic ray in $X$ with the lexicographically least label representing $\xi_i$. Put $p:=\alpha_0$ and for each $i=1,\ldots,K$, put $p_i$ to be a geodesic ray in $X$ from 1 which eventually coincides with $g_i \alpha_i$ (c.f.\ Lemma \ref{lem: moving geodesic basepoint}). %Let $p_i$ be the common terminal subray of $q_i$ and $g_i \alpha_i$. 
    
    Fix a sequence $(r_i)_i$ of subpaths of $p$ as in Lemma \ref{lem:geod ray converging in Y} (3). Denote $\Theta_i$ a fixed relator containing \textcolor{black}{$r_i$ or the subpath $r_i$ itself} if $r_i$ is not contained in a relator.

    By Lemma \ref{lem: form of geodesics in bundle}, we have that for each $i = 1,\ldots,K$, $p_i$ is contained in $\cup_{n \in \N} \Theta_n$ and $p_i \cap \Theta_n \neq \emptyset$ for each $n$. Since each $g_i \alpha_i$ eventually coincides with $p_i$, there exists $N_i \in \N$ such that $g_i \alpha_i$ is eventually contained in $\cup_{n \geq N_i} \Theta_{n}$, and each $g_i \alpha_i$ intersects all of the subgraphs $\Theta_n$ for $n \geq N_i$. 

    \begin{figure}[h]
        \centering
        \includegraphics[width=0.55\linewidth]{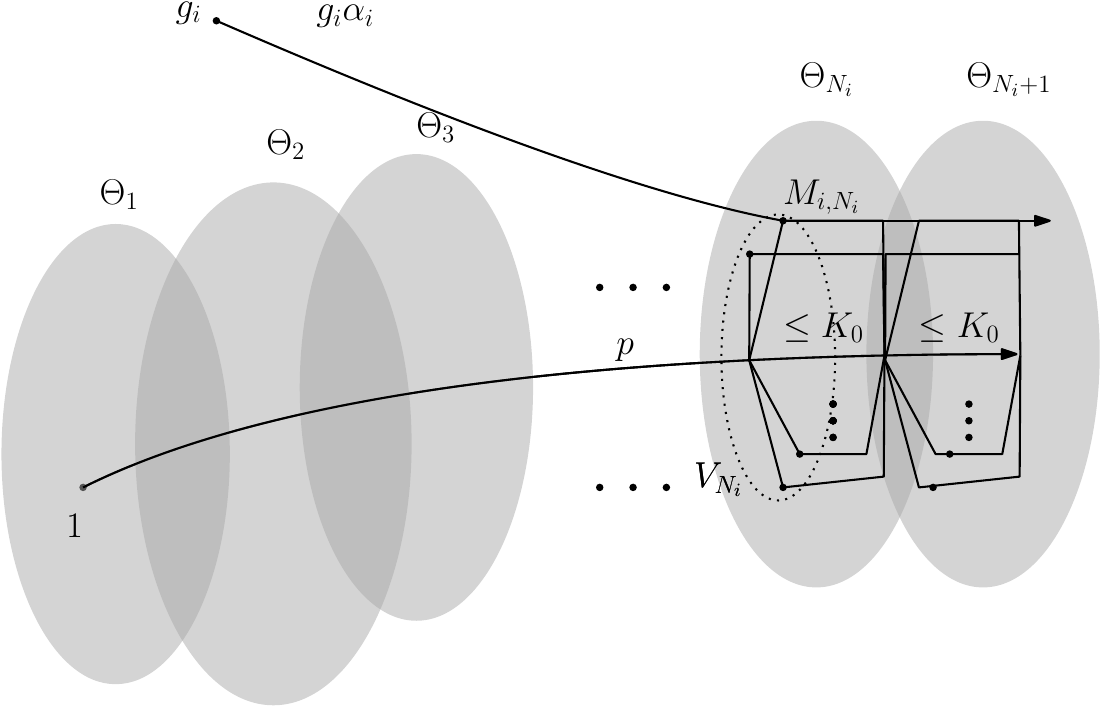}
        \caption{Each geodesic ray $p_i$ (which eventually coincides with $g_i \alpha_i$) eventually passes through all subgraphs $\Theta_n$ for $n \geq N_i$. It can enter each $\Theta_n$ via a set $V_n$ of \textcolor{black}{at most $1+K_0$ vertices.}}
        \label{fig:bounded page entry points}
    \end{figure}

    Let $N = \max \{N_i : i = 1,\ldots,K\}$, so that each geodesic ray $g_i \alpha_i$ eventually passes through $\Theta_n$ for all $n \geq N$, either traversing $\Theta_n$ through a sequence of contours along $p$, or coinciding with $p$. 

    We will show that for each $n \geq N$, we have $\vert \{(p_i \cap \Theta_n)_- : i = 1,\ldots,K\} \vert \leq 1+K_0$.

    %Given a relator $\Theta_i$ for $i \in \mathcal{A}$, define the \textbf{core} of $\Theta_i$ to be the subsegment $s_i$ of $\Theta_i \cap p$ obtained by removing from $\Theta_i \cap p$ its initial and terminal segments of length $\lfloor \lambda \girth(\Theta_i) \rfloor$. Note that $\vert s_i \vert > 0$ since $i \in \mathcal{A}$, so that $\vert \Theta_i \cap p \vert > 2 \lambda \girth(\Theta_i)$. 
    %since no three $\Theta_i$ can pairwise intersect (c.f.\ \cite[Remark 3.7]{Gruber_Sisto}), since $(\Theta_i)_i$ is a geodesic sequence. 
    %Note if $p_i$ passes through a contour $r$ contained in a relator $\Theta_n$, then $r$ must contain $s_{n-1}$ on $p$, since $p_i$ must enter $\Theta_n$ through a vertex in $\Theta_{n-1} \cap \Theta_n$ and exit $\Theta_n$ in a vertex of $\Theta_n \cap \Theta_{n+1}$, and we have for each $i \in \{n-1,n+1\}$ that $\vert \Theta_n \cap \Theta_i \vert < \lambda \girth(\Theta_n)$, so that the segments $\Theta_n \cap \Theta_i \cap p$ are included in the segments of $p$ that we remove from $\Theta_n \cap p$ to define the core. 

    For each $n$, let $s_{n-1}$ be the edge of the finite connected component of $p \setminus \Theta_n$, which is a path, that is at greatest distance in $X$ from 1 in the path. We either have that $(\Theta_n \cap p_i)_- = (\Theta_n \cap p)_-$ (i.e.\ $p_i$ enters $\Theta_n$ through $p$; see Figure \ref{fig:no contours in n-1}) or $(\Theta_n \cap p_i)_- \in r$ for some contour $r$ in $\Theta_{n-1}$ such that $r \setminus p$ is a path (see Figure \ref{fig:entry on r}).%and $r$ is the last contour in $\Theta_{n-1}$ that $p_i$ traverses. 

    \begin{figure}[h]
        \centering
        \includegraphics[width=0.33
        \linewidth]{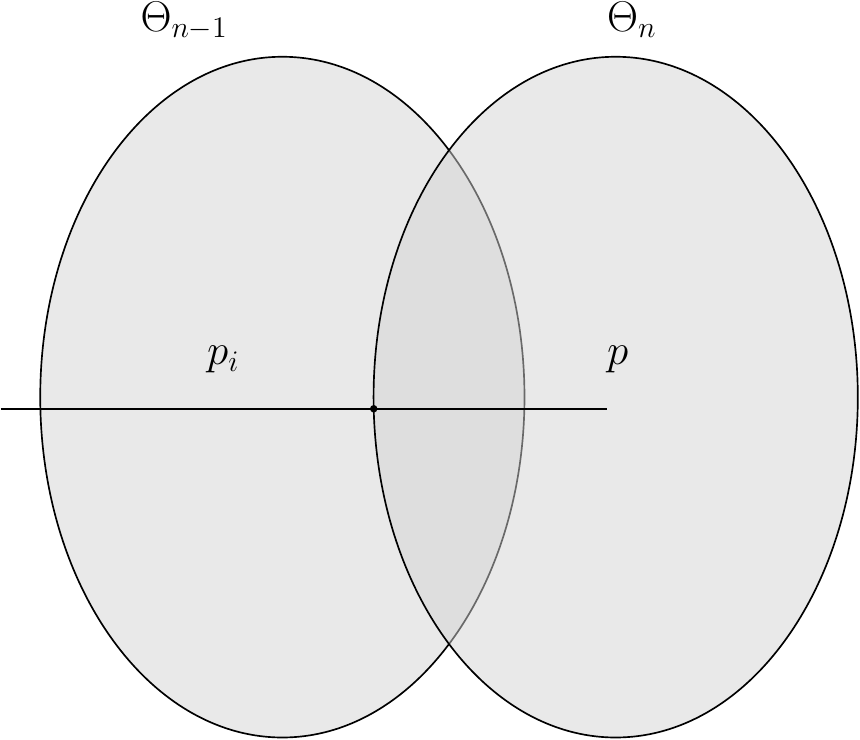}
        \caption{The case when $p_i$ enters $\Theta_{n}$ along $p$.}
        \label{fig:no contours in n-1}
    \end{figure}
    
    Indeed, by Lemma \ref{lem: form of geodesics in bundle}, inside $\Theta_{n-1}$ we have that $p_i$ and $p$ form a sequence of bigon diagrams (see Figure \ref{fig:entry on r}). %Each diagram contained in $\Theta_{n-1}$ consists of a single contour, since we assume that the diagrams are of minimal area. 

    \begin{figure}[h]
        \centering
        \includegraphics[width=0.33\linewidth]{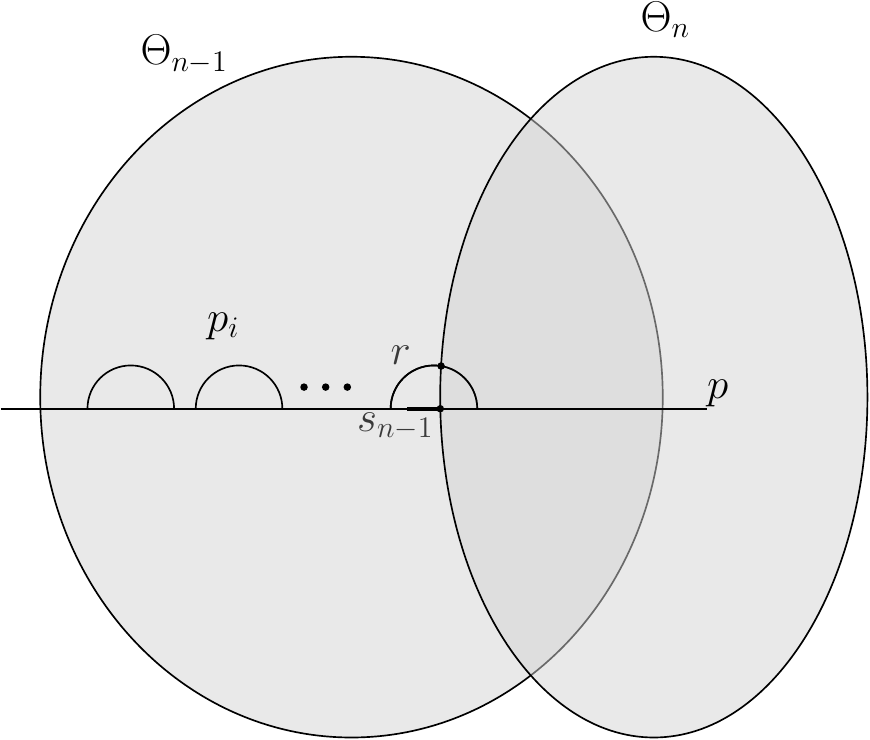}
        \caption{The case when $p_i$ enters $\Theta_n$ through a contour $r$.}
        \label{fig:entry on r}
    \end{figure}

    % \begin{figure}[H]
    %     \centering
    %     \includegraphics[width=0.5\linewidth]{Contours in n-1.pdf}
    %     \caption{The sequence of bigon diagrams formed by $p_i$ and $p$ consisting of single contours in $\Theta_{n}$.}
    %     \label{fig:contours in n-1}
    % \end{figure}
    
    Therefore, if $p_i$ does not enter $\Theta_{n}$ along $p$, then it enters $\Theta_n$ along a bigon diagram, %(which is the last bigon diagram formed by $p$ and $p_i$ contained in $\Theta_{n-1}$)
     hence through a contour $r \subset \Theta_{n-1}$, which is the last contour in $\Theta_{n-1}$ that $p_i$ traverses. In this case, since $\Theta_{n-1} \cap p \supset r \cap p$ and since $\vert r \cap p \vert > (\frac12 - 2 \lambda) \vert r \vert > 3 \lambda \girth(\Theta_{n-1})$, by the small cancellation condition, the relator $\Theta_{n-1}$ is the unique relator containing $r$. We will show that $(r \cap p)_+ \in \Theta_n$. Below, we define $((r \setminus p) \cap \Theta_n)_-$ to be {\color{black}the initial vertex of the last connected component of $(r \setminus p) \cap \Theta_n$} (note that $r \setminus p$ since $r$ is a contour in a diagram of shape $I_1$ bounded by $p$ and $p_i$; see Theorem \ref{thm: Strebel's bigon and triangle classification}). We consider the following cases. 

    \begin{enumerate}
        \item The contour $r$ is the only contour in its diagram. Then $(r \cap p)_+ \in p_i$. If $(\Theta_n \cap p)_-$ occurs after $(r \cap p)_+$ along $p$, then since there are no further bigon diagrams bounded by $p_i$ and $p$ past $r$, we have that $(\Theta_n \cap p)_- \in p_i$, and hence the segment between $((r \setminus p) \cap \Theta_n)_-$ and $(\Theta_n \cap p)_-$ is contained inside $p_i$ and contains $(r \cap p)_+$. By convexity of $\Theta_n$, this segment lies in $\Theta_n$, and hence $(r \cap p)_+ \in \Theta_n$. See Figure \ref{fig:one contour in diagram}.

        \begin{figure}[h]
            \centering
            \includegraphics[width=0.33\linewidth]{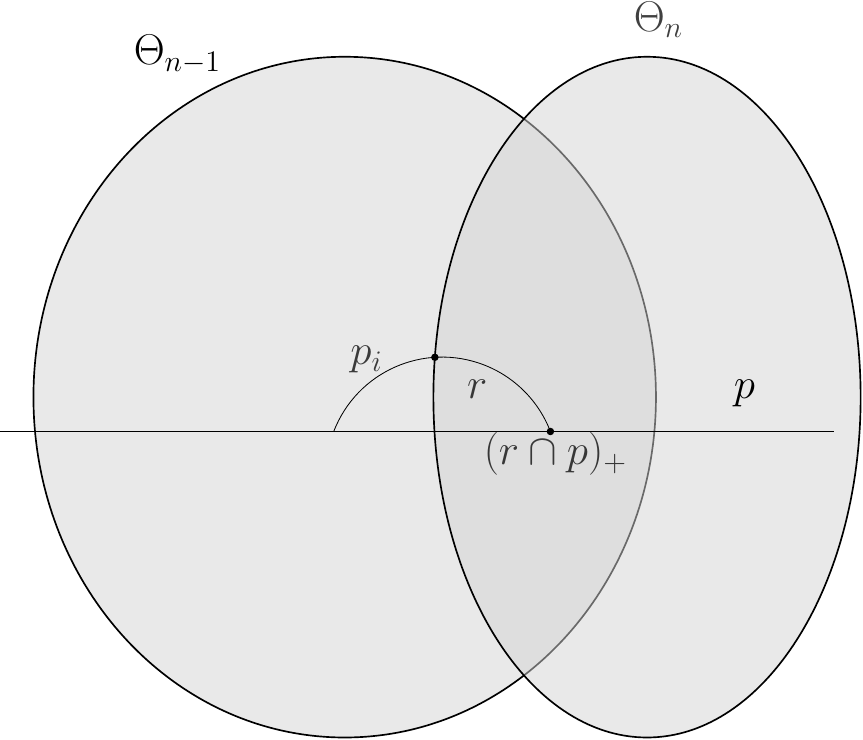}
            \caption{The case when the diagram formed by $p_i$ and $p$ consists of a single contour $r \subset \Theta_{n-1}$. In this case, by convexity of $\Theta_n$, we must have $(r \cap p)_+ \in \Theta_n$.}
            \label{fig:one contour in diagram}
        \end{figure}
        
        \item There exists a contour $r' \subset \Theta_n$ following $r$ in the same diagram. Then $(r \cap p)_+ = (r' \cap p)_- \in \Theta_n$. See Figure \ref{fig:other contour in diagram}.

        \begin{figure}[h]
            \centering
            \includegraphics[width=0.33\linewidth]{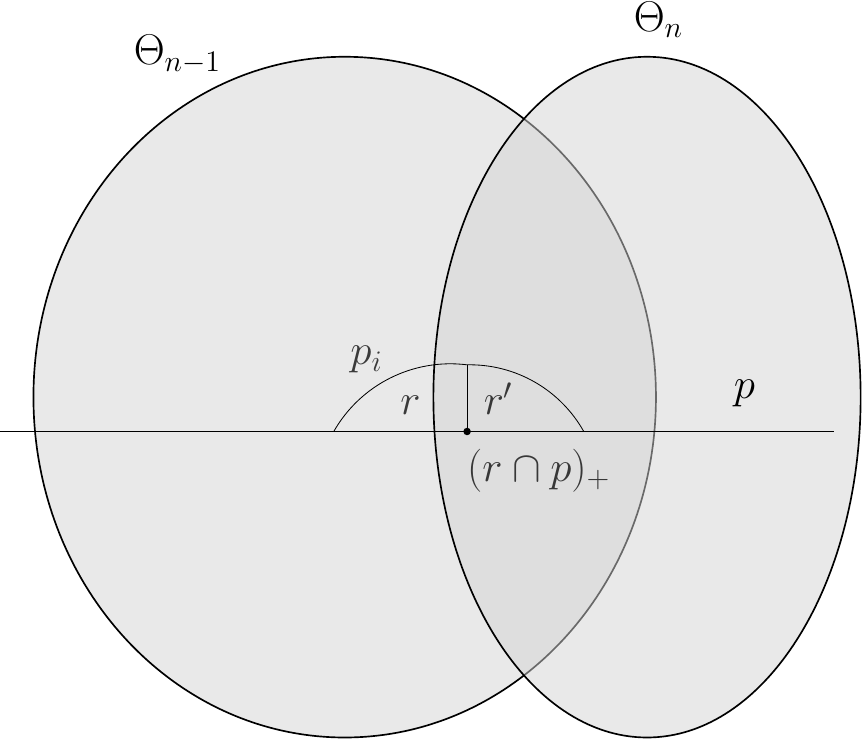}
            \caption{The case when the diagram formed by $p_i$ and $p$ consists of the contour $r \subset \Theta_{n-1}$ and another contour $r' \subset \Theta_n$.}
            \label{fig:other contour in diagram}
        \end{figure}
    \end{enumerate}

    Since $(r \cap p)_+ \in \Theta_n$, we have that $(\Theta_n \cap p)_-$ occurs before $(r \cap p)_+$ on $p$. We cannot have $(\Theta_n \cap p)_-$ occurring before $(r \cap p)_-$ on $p$, since then $r \cap p \subset \Theta_n \cap p$, so that $r \cap p$ is a piece of $\Theta_{n-1}$ and $\Theta_n$, but $\vert r \cap p \vert > \lambda \girth(\Theta_n)$, contradicting the small cancellation assumption. We conclude that $(r \cap p)_-$ must occur before $(\Theta_n \cap p)_-$ and hence $r \cap p$ must contain $s_{n-1}$. 

    In summary, we have shown that 
    $$(p_i \cap \Theta_n)_- \subset \{(p \cap \Theta_n)_-\} \cup \{((r \setminus p) \cap \Theta_n)_- : \text{$r$ is a contour with } s_{n-1} \subset r \subset \Theta_{n-1} \text{ such that $r\setminus p$ is a path}\},$$ 
    {\color{black}where $((r \setminus p) \cap \Theta_n)_-$ is defined to be the initial vertex of the last connected component of $(r \setminus p) \cap \Theta_n$.} By extreme fineness, the latter set has cardinality at most $1+K_0$.

    Thus, the number of points through which each $p_i$ (and hence, $g_i \alpha_i$), can enter $\Theta_n$ is at most $1+K_0$. For each $n \geq N$, let $V_n$ be the set of at most  $1+K_0$ vertices through which a geodesic ray $p_i$ is allowed to enter $\Theta_n$.

    For each $i = 1,\ldots,K$ and $n \geq N$, let $M_{i,n}$ denote the vertex $(p_i \cap \Theta_n)_-$ on $p_i$ through which $p_i$ enters $\Theta_n$.

\begin{figure}[h!]
    \centering
    \includegraphics[width=0.6\linewidth]{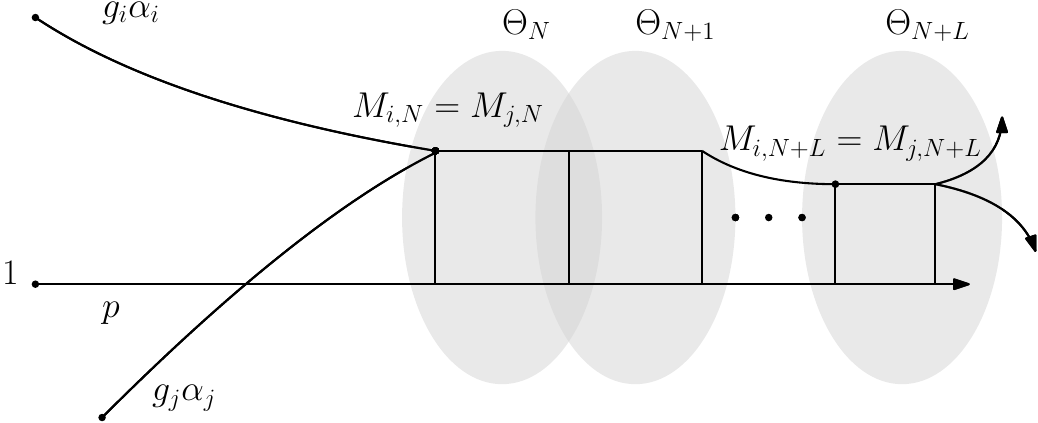}
    \caption{The rays $g_i \alpha_i$ and $g_j \alpha_j$ enter the subgraphs $\Theta_N$ and $\Theta_{N+L}$ through the same vertices $M_{i,N} = M_{j,N}$ and $M_{i,N+L} = M_{j,N+L}$, respectively.}
    \label{fig:geodesic rays in finite index proof}
\end{figure}

    Since the $\xi_i$ are pairwise $R_t$-inequivalent, there exists $L > 0$ such that the labels of the segments of each $p_i$ between $M_{i,N}$ and $M_{i, N+L}$ are pairwise distinct. Since there are at most $1+K_0$ vertices in $V_n$, there are at most $(1+K_0)^2$ choices for possible pairs of vertices $(v_N, v_{N+L}) \in V_N \times V_{N+L}$ through which geodesic rays $p_i$ can enter $\Theta_N$ and $\Theta_{N+L}$. Since $K > (1+K_0)^2$, by the Pigeonhole principle, there exist $i \neq j$ such that $M_{i,N} = M_{j,N}$ and $M_{i,N+L} = M_{j,N+L}$. Since \textcolor{black}{$\alpha_i, \alpha_j$ are geodesic rays with lexicographically least labels}, it follows that the segments on these geodesic rays between the intersection points $M_{i,N} = M_{j,N}$ and $M_{i,N+L} = M_{j,N+L}$ are the same (see Figure \ref{fig:geodesic rays in finite index proof}), contradicting the choice of $L$. 

    Thus, two $\xi_i$ must be $R_t$-equivalent.

\end{proof}
We now conclude the proof of our main theorem. 
\hypersetup{linkcolor=black}
\begin{proof} [Proof of the \hyperlink{Main Theorem}{Main Theorem}]
\hypersetup{linkcolor=red}
        Using the notation above, we have that $R_t \subset E_G$ and by Proposition \ref{prop: finite index} each $E_G$-class contains only finitely many $R_t$-classes. Since $R_t$ is hyperfinite, by Proposition \ref{prop:JKL}, we have that $E_G$ is hyperfinite. 
\end{proof}

	\bibliographystyle{plain} 
	\bibliography{refs}

\iffalse
\vspace{5mm}

\noindent CHRIS KARPINSKI

\noindent  Department of Mathematics and Statistics, McGill University, Burnside Hall,

\noindent 805 Sherbrooke Street West, Montreal, QC, H3A 0B9, Canada.

\noindent E-mail: \emph{christopher.karpinski@mail.mcgill.ca}

\vspace{3mm}

\noindent DAMIAN OSAJDA

\noindent  Institut for Matematiske Fag, University of Copenhagen, 2100 Copenhagen, Denmark.\\
\noindent Instytut Matematyczny, Uniwersytet Wroclawski
\noindent 
pl. Grunwaldzki 2/4, 50–384 Wroclaw, Poland.

\noindent E-mail: \emph{dosaj@math.uni.wroc.pl}

\vspace{3mm}

\noindent KOICHI OYAKAWA

\noindent  Department of Mathematics and Statistics, McGill University, Burnside Hall,

\noindent 805 Sherbrooke Street West, Montreal, QC, H3A 0B9, Canada.

\noindent E-mail: \emph{koichi.oyakawa@mail.mcgill.ca}
\fi

\end{document}